\documentclass[11pt, letterpaper]{amsart}
\usepackage[margin=1.5in]{geometry}

\usepackage[english]{babel}
\usepackage[utf8]{inputenc}
\usepackage[T1]{fontenc}
\usepackage{lmodern} % Fixes bitmap fonts of the previous command.

\usepackage{amsfonts} %May be unncesessary if loading amssymb
\usepackage{amssymb}
\usepackage{amsmath}
\usepackage{amsthm}
\usepackage{verbatim}
\usepackage{enumerate}
\usepackage{mathtools}
\usepackage[shortcuts]{extdash}

\usepackage{tikz-cd}

\usepackage[colorlinks=true, allcolors=blue]{hyperref}

\newtheorem{theorem}{Theorem}[section]
\newtheorem{proposition}[theorem]{Proposition}
\newtheorem{corollary}[theorem]{Corollary}
\newtheorem{lemma}[theorem]{Lemma}

\newtheorem*{theorem-non}{Theorem}

\newtheorem{problemintro}{Problem}

\newtheorem{thmintro}{Theorem}

\theoremstyle{definition}
\newtheorem{definition}[theorem]{Definition}
\newtheorem{example}[theorem]{Example}
\newtheorem{remark}[theorem]{Remark}
\newtheorem*{nota-non}{Notation}

\newcommand{\Z}{\mathbb{Z}}

\newcommand{\C}{\mathbb{C}}

\newcommand{\T}{\mathbb{T}}

\newcommand{\norm}[1]{ \left\| #1 \right\| }

\newcommand{\setbuilder}[2] { \left\{ #1 \enskip \middle| \enskip #2 \right\} }

\newcommand{\card}[1]{\left| #1 \right|}

\newcommand{\set}[1]{\left\{#1\right\}}

\newcommand{\isoto}{\cong}

\newcommand{\surjectsonto}{\twoheadrightarrow}
\newcommand{\tensor}{\otimes}

\newcommand{\B}{\mathrm{B}}
\newcommand{\rC}{\mathrm{C}}
\newcommand{\ca}{\mathrm{C}^*}
\newcommand{\rZ}{\mathrm{Z}}

\DeclareMathOperator{\id}{id}

\DeclareMathOperator{\Ad}{Ad}

\title[Intermediate subalgebras for reduced crossed products]{Intermediate subalgebras for reduced crossed products of discrete groups}

\author{Matthew Kennedy}
\address{Department of Pure Mathematics, University of Waterloo \\ 200 University Avenue West \\ Waterloo, Ontario, N2L 3G1 \\ Canada}
\email{matt.kennedy@uwaterloo.ca}

\author{Dan Ursu}
\address{Mathematisches Institut, Fachbereich Mathematik und Informatik, Universit{\"a}t M{\"u}nster \\ Einsteinstrasse 62 \\ 48149 M{\"u}nster \\ Germany}
\email{dursu@uni-muenster.de}


\thanks{The first author acknowledges the support of the Natural Sciences and Engineering Research Council of Canada (NSERC). The second author was funded by the Deutsche Forschungsgemeinschaft (DFG, German Research Foundation) -- Project-ID 427320536 -- SFB 1442, as well as under Germany's Excellence Strategy EXC 2044 390685587, Mathematics M{\"u}nster: Dynamics--Geometry--Structure, and by the ERC Advanced Grant 834267 -- AMAREC}

\begin{document}
    \begin{abstract}
         Let $\alpha : \Gamma \curvearrowright A$ be an action of a discrete group $\Gamma$ on a unital C*-algebra $A$ by *-automorphisms and let $A \rtimes_{\alpha,\lambda} \Gamma$ denote the corresponding reduced crossed product C*-algebra. Assuming that $\Gamma$ satisfies the approximation property, we establish a sufficient and (almost always) necessary condition on the action $\alpha$ for the existence of a Galois correspondence between intermediate C*-algebras for the inclusion $A \subseteq A \rtimes_{\alpha,\lambda} \Gamma$ and partial subactions of $\alpha$. This condition, which we refer to as pointwise residual proper outerness, is a natural noncommutative generalization of freeness.
    \end{abstract}
 
	\maketitle
	
	\tableofcontents
	
	\section{Introduction}
    \label{sec:introduction}
    
    Many interesting and important problems about operator algebras can be expressed in terms of the structure of intermediate operator algebras for an inclusion $A \subseteq B$ of two relatively well understood operator algebras $A$ and $B$. An intermediate operator algebra for this inclusion is an operator algebra $C$ satisfying
    \[ A \subseteq C \subseteq B. \]
    In this paper, we will obtain definitive results on one of the most well-studied variants of this problem, which arises from the study of group actions on C*-algebras.
    
    Let $\alpha : \Gamma \curvearrowright A$ be an action of a discrete group $\Gamma$ on a unital C*-algebra $A$ by *-automorphisms. We may construct the corresponding reduced crossed product C*-algebra $A \rtimes_{\alpha,\lambda} \Gamma$, which is generated by $A$ and a regular representation of $\Gamma$, and encodes the action $\alpha$. This gives rise to the inclusion of C*-algebras
    \[ A \subseteq A \rtimes_{\alpha,\lambda} \Gamma,\]
    which is the main subject of consideration in this paper.

    One way to obtain an intermediate C*-algebra for this inclusion is to consider a restriction $\alpha|_\Lambda : \Lambda \curvearrowright A$ of the action $\alpha$ to a subgroup $\Lambda \subseteq \Gamma$. The corresponding reduced crossed product $A \rtimes_{\alpha|_{\Lambda}, \lambda} \Lambda$ satisfies
    \[ A \subseteq A \rtimes_{\alpha|_{\Lambda}, \lambda} \Lambda \subseteq A \rtimes_{\alpha,\lambda} \Gamma. \]

    More generally, an intermediate C*-algebra for the inclusion can be obtained from a \emph{partial subaction} of the action $\alpha$. This is a partial action $\beta : \Gamma \curvearrowright A$ of $\Gamma$ on $A$ by partial *-automorphisms in the sense of e.g. \cite{exel_partial_dynamical_systems}, with the additional requirement that each partial *-automorphism $\beta_g$ is a restriction of the *-automorphism $\alpha_g$. A canonical copy of the corresponding reduced crossed product $A \rtimes_{\beta,\lambda} \Gamma$ satisfies
    \[ A \subseteq A \rtimes_{\beta,\lambda} \Gamma \subseteq A \rtimes_{\alpha,\lambda} \Gamma. \]
    It is natural to ask when every intermediate C*-algebra has this form, which is equivalent to the existence of a Galois correspondence between intermediate subalgebras and partial subactions.

    \begin{problemintro} \label{problemintro:main}
        Let $\alpha : \Gamma \curvearrowright A$ be an action of a discrete group $\Gamma$ on a unital C*-algebra $A$ by *-automorphisms. Find necessary and sufficient conditions so that every intermediate C*-algebra for the inclusion $A \subseteq A \rtimes_{\alpha,\lambda} \Gamma$ is the reduced crossed product of a partial subaction of $\alpha$.
    \end{problemintro}

    We will provide an (essentially) complete solution to Problem \ref{problemintro:main}. Specifically, we will establish a sufficient condition for the existence of a Galois correspondence between intermediate subalgebras and partial subactions. We will then prove that this condition is almost always necessary, in a precise sense. 
    
    Before stating our results, we briefly pause to discuss the history surrounding this problem. Since there is an enormous body of related literature dating back at least 45 years, we will highlight only the most relevant results.
    
    In the commutative setting when $A$ is a C*-algebra $\rC(X)$ of continuous functions on a compact Hausdorff space $X$, the best result was obtained by Brown, Exel, Fuller, Pitts, and Reznikoff in \cite[Corollary~5.8]{befpr_groupoid_intermediate}. Their result asserts that if $\alpha : \Gamma \curvearrowright \rC(X)$ is a free action of a discrete group with the approximation property on a unital commutative C*-algebra $\rC(X)$, then every intermediate C*-algebra for the inclusion $\rC(X) \subseteq \rC(X) \rtimes_{\alpha,\lambda} \Gamma$ is the reduced crossed product of a partial subaction of $\alpha$.

    In the noncommutative setting, the best result was obtained by Cameron and Smith \cite{cameron_smith_cstar_simple}. They prove that if $\alpha : \Gamma \curvearrowright A$ is an outer action of a discrete group on a unital simple C*-algebra $A$, then every intermediate C*-algebra for the inclusion $A \subseteq A \rtimes_{\alpha,\lambda} \Gamma$ is the reduced crossed product of a restriction of the action $\alpha$ to a subgroup of $\Gamma$. This extended previous results of Landstad, Olesen and Pedersen \cite[Theorem~3.1]{LOP1978} for the case when $\Gamma$ is abelian and Izumi \cite[Corollary~6.6]{Izu2002} for the case when $\Gamma$ is finite.

    The von Neumann-algebraic version of Problem \ref{problemintro:main} has also received extensive consideration. Let $\alpha : \Gamma \curvearrowright M$ be an action of a discrete group $\Gamma$ on a von Neumann algebra $M$ by *-automorphisms. Constructing the von Neumann algebra crossed product $M \overline{\rtimes}_\alpha \Gamma$ gives rise to the inclusion of von Neumann algebras $M \subseteq M \overline{\rtimes}_\alpha \Gamma$, and one can pose the analogue of Problem \ref{problemintro:main}. In this setting, the best results were obtained by Cameron and Smith \cite{cameron_smith_vn_arbitrary}. Under the assumption that $\Gamma$ has the approximation property and the action $\alpha$ is properly outer, they obtain a characterization of intermediate von Neumann algebras for the inclusion that is analogous to the desired characterization in Problem~\ref{problemintro:main}. This extended previous results of Cameron and Smith \cite[Theorem~4.4]{cameron_smith_vn_factor} for the case when $M$ is a factor, Izumi, Longo and Popa \cite[Theorem~3.13]{ILP1998} for the case when $M$ is a separable factor, and Choda \cite{Cho1978} for the case when $M$ is a separable factor and each intermediate von Neumann algebra is the range of a faithful normal conditional expectation.

    % Are there any results on the sufficient part? 

    The above results suggested to us that a solution to Problem \ref{problemintro:main} would require the action $\alpha : \Gamma \curvearrowright A$ to satisfy some generalized notion of freeness for group actions on noncommutative C*-algebras. There are a number of such definitions in the literature, each with different strengths and weaknesses (see e.g.\ \cite{kwasniewski_meyer_aperiodicity}). For example, spectral freeness, outerness and proper outerness, to name just a few. One can also consider residual versions of these properties, which means that they pass to quotients by invariant ideals. 

    The main result in \cite{cameron_smith_cstar_simple} requires $\alpha$ to be outer. However, it also requires the simplicity of $A$, and it is known that for simple C*-algebras, many of the above definitions coincide. For non-simple C*-algebras, where this is no longer the case, the situation is much more complicated. Since Problem~\ref{problemintro:main} asks for both necessary and sufficient conditions, settling on the correct definition presented a major challenge.

    It turns out that the correct definition is something we call \emph{pointwise residual proper outerness}, which means that $\alpha_g$ is residually properly outer for every $g \in \Gamma \setminus \{e\}$. If $A$ is commutative, then pointwise residual proper outerness is equivalent to the freeness of the induced action on the spectrum of $A$. See Section~\ref{sec:properly_outer} for details.

    We are now able to state our first main result, which establishes sufficient conditions for a Galois correspondence between intermediate C*-algebras and partial subactions. We prove this result in Section~\ref{sec:free_action_correspondence}. 

    \begin{thmintro}
    \label{thmintro:main}
        Let $\alpha : \Gamma \curvearrowright A$ be an action of a discrete group $\Gamma$ with the approximation property on a unital C*-algebra $A$ by *-automorphisms. If either $\card{\Gamma} \leq 2$ or $\alpha$ is pointwise residually properly outer, then every intermediate C*-algebra for the inclusion $A \subseteq A \rtimes_{\alpha,\lambda} \Gamma$ is the reduced crossed product of a partial subaction of $\alpha$.
    \end{thmintro}

    The requirement that $\Gamma$ has the approximation property is standard in this setting, and is generally required whenever $A$ is not assumed to be simple. We address this point in detail in Section~\ref{sec:approximation_property}. 

    The key technical step in the proof of Theorem~\ref{thmintro:main} is an averaging result to show that if $B$ is an intermediate C*-algebra for the inclusion $A \subseteq A \rtimes_{\alpha,\lambda} \Gamma$ and $b \in B$ has Fourier series $b \sim \sum_{g \in \Gamma} b_g \lambda_g$, then $b_g \lambda_g \in B$ for all $g \in \Gamma$.
    
    Our approach to the proof of this result is motivated by the boundary-theoretic techniques that have recently been developed to study the ideal structure of group C*-algebras and crossed products (see e.g.\ \cite{KK2017,BKKO2017,K2020,KS2019,GU2023}). We also make use of techniques from noncommutative convexity, in the form of a separation result for C*-convex sets that refines a separation result of Magajna \cite{magajna_bimodules}.

    A similar averaging result is proved using much different methods in the more restricted setting of \cite{cameron_smith_cstar_simple}. Their argument utilizes von Neumann-algebraic techniques applied to the bidual of the C*-algebra, and requires an intricate use of disintegration theory.

    Our second main result is a (nearly) complete converse to Theorem \ref{thmintro:main}. Together with Theorem \ref{thmintro:main}, it implies that for most groups, pointwise residual outerness of the action is both necessary and sufficient for every intermediate subalgebra to be the reduced crossed product of a partial subaction. This result is proved in Section \ref{sec:counterexample_constructions}.

    \begin{thmintro}
    \label{thmintro:converse}
        Let $\alpha : \Gamma \curvearrowright A$ be an action of a discrete group $\Gamma$ with $|\Gamma| \geq 3$ on a unital C*-algebra $A$ by *-automorphisms. For an ideal $K \triangleleft A$, let $\Gamma_K$ denote the corresponding rigid stabilizer, i.e.
        \[
            \Gamma_K = \{g \in \Gamma : \alpha_g(K) = K \text{ and } \alpha_g \text{ is quasi-inner on } A/K \}.
        \]
        If $\card{\Gamma_K} \geq 3$ for a proper ideal $K \triangleleft A$, then there is an intermediate C*-algebra for the inclusion $A \subseteq A \rtimes_{\alpha,\lambda} \Gamma$ that is not the reduced crossed product of a partial subaction of $\alpha$.
    \end{thmintro}

    The notion of a rigid stabilizer for a group action on a C*-algebra is an analogue of a point stabilizer for a group acting on a topological space, and the failure of an action to be pointwise residually properly outer is equivalent to the existence of non-trivial rigid stabilizers. See Section \ref{sec:counterexample_constructions:noncommutative} for details.

    The only case missing from Theorem \ref{thmintro:converse} is when every non-trivial rigid stabilizer for the action is of order $2$. It turns out that Theorem \ref{thmintro:converse} fails in this case. We will demonstrate this in Section~\ref{sec:noncommutative_counterexamples}. However, the failure of Theorem~\ref{thmintro:converse} in this special case is a strictly noncommutative phenomenon. In the commutative setting, we will establish a complete converse for Theorem~\ref{thmintro:main} in Section~\ref{thmintro:commutative}. This yields the following refinement of \cite[Corollary~5.8]{befpr_groupoid_intermediate}. 

    \begin{thmintro}
    \label{thmintro:commutative}
         Let $\theta : \Gamma \curvearrowright X$ be an action of a discrete group $\Gamma$ on a compact Hausdorff space $X$ by homeomorphisms and let $\alpha : \Gamma \curvearrowright \rC(X)$ denote the corresponding action of $\Gamma$ on the unital commutative C*-algebra $\rC(X)$ by *-automorphisms. If $\Gamma$ has the approximation property, and either $|\Gamma| \leq 2$ or $\theta$ is free, then every intermediate subalgebra for the inclusion $\rC(X) \subseteq \rC(X) \rtimes_{\alpha,\lambda} \Gamma$ is the reduced crossed product of a partial subaction of $\alpha$. Conversely, if $|\Gamma| \geq 3$ and $\theta$ is not free, then there is an intermediate C*-algebra for the above inclusion that is not the reduced crossed product of a partial subaction of $\alpha$. 
    \end{thmintro}

    Finally, we want to mention two recent papers that the authors found particularly inspiring. The first is Suzuki's \cite{suzuki_2020} paper on intermediate C*-algebras and intermediate von Neumann algebras for another type of inclusion arising from crossed products. The second is R\o{}rdam's \cite{rordam_2023} paper relating averaging properties for C*-algebras with various noncommutative generalizations of freeness.

    \section*{Acknowledgements}
    The authors are grateful to Bartosz Kwa\'{s}niewski for both helpful dicusssions about noncommutative freeness, and giving numerous corrections and suggestions while looking through a draft of this paper. The authors are also grateful to Camila Sehnem for answering a number of questions about partial actions and reduced crossed products of partial actions.

    \section{Preliminaries}
    \label{sec:preliminaries}
    
    \subsection{Actions, partial actions and reduced crossed products}
    \label{sec:dynamical_systems}
    
    In this section we will briefly review some facts about actions and partial actions of groups on C*-algebras, as well as their corresponding reduced crossed products. For a reference on group actions and their reduced crossed products, we refer the reader to the book of Brown and Ozawa \cite{brown_ozawa}. For a reference on partial actions and their reduced crossed products, we refer the reader to the book of Exel \cite{exel_partial_dynamical_systems} and the paper of Quigg and Raeburn \cite{QR}.

    Recall that an action $\alpha : \Gamma \curvearrowright A$ of a discrete group $\Gamma$ on a unital C*-algebra $A$ by *-automorphisms is a group homomorphism $\alpha : \Gamma \to \operatorname{Aut}(A)$. If $A$ is commutative, say $A = \rC(X)$ for a compact Hausdorff space $X$, then it follows from Gelfand duality that $\alpha$ corresponds to an action $\theta : \Gamma \curvearrowright X$ of $\Gamma$ on $X$ by homeomorphisms, with $\alpha_g(f) = f \circ \theta_{g^{-1}}$ for $g \in \Gamma$ and $f \in \rC(X)$. 
    
    \begin{definition}
    A \emph{partial action} $\beta : \Gamma \curvearrowright A$ of a discrete group $\Gamma$ on a unital C*-algebra $A$ by partial *-automorphisms consists of a family of ideals $\{I_g\}_{g \in \Gamma}$ and a family of *-isomorphisms $\beta = \{\beta_g : I_{g^{-1}} \to I_g\}_{g \in \Gamma}$ satisfying
    \begin{enumerate}
        \item $I_e = A$ and $\beta_e = \id_A$,
        \item $\beta_h^{-1}(I_h \cap I_{g^{-1}}) \subseteq I_{(gh)^{-1}}$ and $\beta_{gh} = \beta_g \beta_h$ on $\beta_h^{-1}(I_h \cap I_{g^{-1}})$ for all $g,h \in \Gamma$.
    \end{enumerate}
    If $\alpha : \Gamma \curvearrowright A$ is an action of $\Gamma$ on $A$ by *-automorphisms, then we will say that $\beta$ is a \emph{partial subaction} of $\alpha$ if $\beta_g = \alpha_g|_{I_{g^{-1}}}$ for all $g \in \Gamma$.
    \end{definition}

    If $A$ is commutative, say $A = \rC(X)$ for a compact Hausdorff space $X$, and $\beta : \Gamma \curvearrowright \rC(X)$ is a partial action, then it follows from Gelfand duality that $\beta$ corresponds to a partial action $\eta : \Gamma \curvearrowright X$ of $\Gamma$ on $X$ by partial homeomorphisms. Let $\{I_g\}_{g \in \Gamma}$ denote the family of ideals of $A$ from the definition of $\beta$ and let $\{U_g\}_{g \in \Gamma}$ denote the corresponding family of open subsets of $X$ defined by $U_g = \operatorname{supp}(I_g)$, so that $I_g = \rC_0(U_g)$ for $g \in \Gamma$. The *-isomorphism $\beta_g : \rC_0(U_{g^{-1}}) \to \rC_0(U_g)$ corresponds to the homeomorphism $\eta_{g^{-1}} : U_g \to U_{g^{-1}}$ via $\beta_g(f) = f \circ \eta_{g^{-1}}$ for $g \in \Gamma$ and $f \in \rC_0(U_{g^{-1}})$.

    Let $\alpha : \Gamma \curvearrowright A$ be an action of a discrete group $\Gamma$ on a unital C*-algebra $A$ by *-automorphisms. The corresponding reduced crossed product $A \rtimes_{\alpha,\lambda} \Gamma$ is a C*-algebra generated by a copy of $A$ and a regular representation $\lambda$ of $\Gamma$ that implements $\alpha$, in the sense that
    \[
        \alpha_g(a) = \lambda_g a \lambda_g^* \quad \text{ for all } \quad a \in A,\ g \in \Gamma.
    \]
    It is a convenient fact that elements in the algebraic span of $\{a \lambda_g : a \in A,\ g \in \Gamma\}$ are dense in $A \rtimes_{\alpha,\lambda} \Gamma$. 
    
    A characteristic property of $A \rtimes_{\alpha,\lambda} \Gamma$ is the existence of a faithful conditional expectation $E : A \rtimes_{\alpha,\lambda} \Gamma \to A$ satisfying
    \[
        E(a \lambda_g) = \begin{cases} a & g = e \\ 0 & \text{otherwise} \end{cases}, \quad \text{for} \quad a \in A,\ g \in \Gamma.
    \]
    An element $b \in A \rtimes_{\alpha,\lambda} \Gamma$ is uniquely determined by its Fourier series $b \sim \sum_{g \in \Gamma} b_g \lambda_g$, where $b_g = E(b \lambda_g^*)$ denotes the $g$-th Fourier coefficient of $b$. We note that, as in the classical setting, the Fourier series for $b$ does not necessarily converge to $b$ in any reasonable sense. We will refer to the set $\{g \in \Gamma : b_g \ne 0\}$ of nonzero Fourier coefficients of $b$ as the \emph{support} of $b$, so that $b$ belongs to the algebraic span of $\{a \lambda_g : a \in A,\ g \in \Gamma\}$ if and only if it has finite support.

    Let $\beta : \Gamma \curvearrowright A$ be a partial subaction of $\alpha$. The corresponding reduced crossed product $A \rtimes_{\beta,\lambda} \Gamma$, is typically defined as the reduced cross-sectional C*-algebra of a Fell bundle. However, for the purposes of this paper, it will be convenient to instead realize $A \rtimes_{\beta,\lambda} \Gamma$ as an intermediate C*-algebra for the inclusion $A \subseteq A \rtimes_{\alpha,\lambda} \Gamma$.
    
    Specifically, let $\{I_g\}_{g \in \Gamma}$ denote the family of ideals for $A$ from the definition of $\beta$. In the next result, we will prove that the closed span of
    \[
        \{a_g \lambda_g : a_g \in I_g,\ g \in \Gamma\} \subseteq A \rtimes_{\alpha,\lambda} \Gamma
    \]
    is an intermediate C*-algebra for the inclusion $A \subseteq A \rtimes_{\alpha,\lambda} \Gamma$ that is isomorphic to the standard definition of the reduced crossed product $A \rtimes_{\beta,\lambda} \Gamma$ of $\beta$. We will then be able to identify $A \rtimes_{\beta,\lambda} \Gamma$ with the above C*-algebra.

    \begin{proposition}
    \label{prop:partial_crossed_product_inclusion}
        Let $\alpha : \Gamma \curvearrowright A$ be an action of a discrete group $\Gamma$ on a unital C*-algebra $A$ by *-automorphisms and let $\beta : \Gamma \curvearrowright A$ be a partial subaction of $\alpha$ with corresponding family of ideals $\{I_g\}_{g \in \Gamma}$. The closed span of
        \[
            \{a_g \lambda_g : a_g \in I_g,\ g \in \Gamma\} \subseteq A \rtimes_{\alpha,\lambda} \Gamma
        \]
        is an intermediate C*-algebra for the inclusion $A \subseteq A \rtimes_{\alpha,\lambda} \Gamma$ that is isomorphic to the reduced crossed product $A \rtimes_{\beta,\lambda} \Gamma$.
    \end{proposition}

    \begin{proof}
        Let $B$ denote the closed span of $\{a_g \lambda_g : a_g \in I_g,\ g \in \Gamma\} \subseteq A \rtimes_{\alpha,\lambda} \Gamma$. It is easy to see that $B$ is a C*-algebra, and in particular an intermediate C*-algebra for the inclusion $A \subseteq A \rtimes_{\alpha,\lambda} \Gamma$. We will prove that $B$ is isomorphic to the reduced crossed product $A \rtimes_{\beta,\lambda} \Gamma$ using Quigg and Raeburn's characterization from \cite{QR} of reduced crossed products of partial actions in terms of coactions, making use of their terminology.

        For $g \in \Gamma$, let $p_g \in A^{**}$ denote the central support projection for the ideal $I_g$. Note that $\beta_{g^{-1}}(p_g) = p_{g^{-1}}$. Let $v_g = p_g \lambda_g \in (A \rtimes_{\alpha,\lambda} \Gamma)^{**}$. Then
        \[ v_g^* = \lambda_{g^{-1}} p_g = \beta_{g^{-1}}(p_g) \lambda_{g^{-1}} = p_{g^{-1}} \lambda_{g^{-1}} = v_{g^{-1}}, \]
        so $v_g$ is a partial isometry with initial projection $p_{g^{-1}}$ and final projection $p_g$. Hence $v : \Gamma \to (A \rtimes_{\alpha,\lambda} \Gamma)^{**}$ is a partial representation of $\Gamma$. Furthermore, for $a \in I_{g^{-1}}$,
        \[ v_g a v_g^* = p_g \lambda_g a \lambda_g^* p_g = \beta_g(a) \in I_g, \]
        so $v_g$ implements $\beta_g$.
        
        From above, the pair $(\id_A,v)$ is a covariant representation for the partial action $\beta$, and so gives rise to a representation $\id_A \times v$ of the full crossed product $A \rtimes_\beta \Gamma$ of $\beta$. The C*-algebra generated by this representation is exactly the C*-algebra $B$.
        
        Let $\ca_u(\Gamma)$ denote the full C*-algebra of $\Gamma$ generated by a universal unitary representation $u$ of $\Gamma$. Let $\delta : A \rtimes_{\alpha,\lambda} \Gamma \to A \rtimes_{\alpha,\lambda} \Gamma \otimes \ca_u(\Gamma)$ denote the canonical normal coaction on $A \rtimes_{\alpha,\lambda} \Gamma$ satisfying
        \[
            \delta(a \lambda_g) = a \lambda_g \otimes u_g \quad \text{for} \quad a \in A,\ g \in \Gamma.
        \]
        Note that $\delta$ is the composition of the isomorphism sending $a \lambda_g \to a \lambda_g \otimes \lambda_g \to $
        In particular, the restriction $\delta|_B : B \to B \otimes \ca_u(\Gamma)$ is a coaction on $B$ satisfying
        \[ \delta|_B \circ \id_A \times v = ((\id_A \times v) \otimes \id) \circ \hat{\beta}, \]
        where $\hat{\beta} : A \rtimes_\beta \Gamma \to A \rtimes_\beta \Gamma \otimes \ca_u(\Gamma)$ is the dual coaction for $\beta$. It now follows from \cite[Corollary 3.8]{QR} that $B$ is ismomorphic to $A \rtimes_{\beta,\lambda} \Gamma$.
    \end{proof}
    
    In light of Proposition~\ref{prop:partial_crossed_product_inclusion}, we are now able to make the following definition.
    
    \begin{definition}
        Let $\alpha : \Gamma \curvearrowright A$ be an action of a discrete group $\Gamma$ on a unital C*-algebra $A$ by *-automorphisms and let $\beta : \Gamma \curvearrowright A$ be a partial subaction of $\alpha$ with corresponding family of ideals $\{I_g\}_{g \in \Gamma}$. The reduced crossed product $A \rtimes_{\beta,\lambda} \Gamma$ of $\beta$ is the closed span of
        \[
            \{a_g \lambda_g : a_g \in I_g,\ g \in \Gamma\} \subseteq A \rtimes_{\alpha,\lambda} \Gamma.
        \]
    \end{definition}

    \subsection{Groups with the approximation property}
    \label{sec:approximation_property}
    In this section we will briefly discuss the \emph{approximation property} for discrete groups, which was introduced in the work of Haagerup and Kraus \cite{haagerup_kraus_approximation_property}. We take the following equivalent characterization of the approximation property as the definition (see e.g.\ \cite[Theorem~12.4.9]{brown_ozawa} and \cite[Definition~2.3]{suzuki_decreasing_intersection}).

    \begin{definition}
    \label{def:approximation_property}
        A discrete group $\Gamma$ is said to have the \emph{approximation property} if there is a net $\varphi_i : \Gamma \to \mathbb{C}$ of finitely supported functions such that the maps
        \[ m_{\varphi_i} \tensor \id_{\B(\ell^2)} : C^*_\lambda(\Gamma) \tensor_{\min} \B(\ell^2) \to C^*_\lambda(\Gamma) \tensor_{\min} \B(\ell^2) \]
        converge to the identity map in the point-norm topology, where $\ell^2$ denotes the unique separable infinite dimensional Hilbert space and $m_{\varphi} : \ca_\lambda(\Gamma) \to \ca_\lambda(\Gamma)$ is the unique bounded linear map satisfying $m_{\varphi}(\lambda_g) = \varphi(g) \lambda_g$.
    \end{definition}

    The following application of the approximation property will be of great importance to us. It follows from the proof of \cite[Proposition~3.4]{suzuki_decreasing_intersection}. We also note that this proof also holds for uncountable groups and non-separable C*-algebras, as it is an easy exercise that Definition~\ref{def:approximation_property} will hold when replacing $\B(\ell^2)$ with $\B(H)$ for a non-separable Hilbert space $H$. 

    \begin{proposition}
    \label{prop:ap_crossed_products}
        Assume $\Gamma$ is a discrete group with the approximation property, and that $\alpha : \Gamma \curvearrowright A$ is an action of $\Gamma$ on a unital C*-algebra $A$ by *-automorphisms. Then given any element $b \in A \rtimes_{\alpha,\lambda} \Gamma$ with Fourier series $b \sim \sum_{g \in \Gamma} b_g \lambda_g$, we have that $b$ belongs to the closed span of $\{b_g \lambda_g : g \in \Gamma\}$.
    \end{proposition}

    Many groups are known to have the approximation property. For example, Haagerup and Kraus proved that this is true for weakly amenable groups. This class of groups includes all amenable groups and, by a result of Ozawa \cite{ozawa_2008}, all hyperbolic groups. However, Lafforgue and de la Salle \cite{LS2011} proved that $\mathrm{SL}_n(\mathbb{Z})$ does not have the approximation property for $n \geq 3$.

    \subsection{The injective envelope and pointwise residual proper outerness}
    \label{sec:properly_outer}
    
    In this section we will briefly review some background material about group actions on C*-algebras and injective envelopes before introducing the definition of a pointwise residually properly outer action.

    Recall that an action $\theta : \Gamma \curvearrowright X$ of a discrete group $\Gamma$ on a compact Hausdorff space $X$ by homeomorphisms is (pointwise) \emph{free} if for every $g \in G \setminus \{e\}$, the homeomorphism $\theta_g$ is free, i.e. the fixed point set of $\theta_g$ is empty. The action $\theta$ is (pointwise) \emph{topologically free} if for every $g \in \Gamma \setminus \{e\}$, the homeomorphism $\theta_g$ is topologically free, i.e.\ the fixed point set of $\theta_g$ has empty interior. 
    
    It is easy to verify that $\theta$ is free if and only if it is ``residually topologically free'', in the sense that for every $g \in G \setminus \{e\}$ and every closed subset $X_0 \subseteq X$ with $\theta_g(X_0) = X_0$, the restriction $\theta_g|_{X_0}$ is topologically free. 

    As mentioned in the introduction, we require an appropriate generalization of (pointwise) freeness for an action of a discrete group on a potentially noncommutative C*-algebra. Kishimoto \cite{Kishimoto_freely_acting} introduced a very natural generalization of topological freeness for *-automorphisms called \emph{proper outerness}, which has proven to be especially valuable in combination with the boundary-theoretic techniques mentioned in the introduction. Motivated by this, as well as the observation above that freeness can be characterized in terms of topological freeness on closed invariant subsets, our noncommutative generalization of (pointwise) freeness for an action, which we will call pointwise residual proper outerness, will be defined in terms of the proper outerness of the individual *-automorphism on invariant quotients.

    There are a number of equivalent characterizations of the proper outerness of a *-automorphism of a C*-algebra. It will be convenient for our purposes to follow Hamana \cite{hamana85-injective_envelopes_equivariant} and define it in terms of the injective envelope of the C*-algebra. Therefore, before defining proper outerness and pointwise residual proper outerness, we will briefly review some basic facts about the injective envelope of a C*-algebra from \cite{hamana79_injective_envelopes_cstaralg,hamana79_injective_envelopes_opsys}.

    For a unital C*-algebra $A$, the \emph{injective envelope} $I(A)$ is the smallest (in a precise sense) injective C*-algebra with the property that there is an injective *-homomorphism from $A$ into $I(A)$. It always exists and is unique up to isomorphism. If $A$ is non-unital, then the injective envelope of $A$ is defined to be the injective envelope of its minimal unitization. It is convenient to identify $A$ with its image in $I(A)$.

    The inclusion $A \subseteq I(A)$ is \emph{essential}, meaning that any unital completely positive map $\phi : I(A) \to B$ such that $\phi|_A$ is a complete isometry is a complete isometry. It is also \emph{rigid}, meaning that the only unital completely positive map $\phi : I(A) \to I(A)$ such that $\phi|_A = \id_A$ is $\id_{I(A)}$. If $C$ is an injective C*-algebra and the inclusion $A \subseteq C$ is rigid, then $C = I(A)$. An important consequence of the rigidity of the inclusion $A \subseteq I(A)$ is that every *-automorphism of $A$ has a unique extension to a unital completely positive map on $I(A)$ that is also a *-automorphism.
    
    The injective envelope $I(A)$ is typically not a von Neumann algebra, but like von Neumann algebras it is monotone complete, meaning that every bounded increasing net of self-adjoint elements in $I(A)$ admits a supremum. More generally, it is known that every injective C*-algebra is monotone complete - see, for example, \cite[Corollary~8.1.5]{SWMonotoneComplete}.

    Hamana observed in \cite[Proposition 5.1]{hamana82_mc_tensor_products_I} that Kallman's \cite{kallman_free_actions} decomposition of a *-automorphism of a von Neumann algebra into inner and properly outer parts applies more generally to monotone complete C*-algebras. Let $A$ be a monotone complete C*-algebra and let $\alpha : A \to A$ be a *-automorphism. Then there is a largest $\alpha$-invariant central projection $p \in A$ such that the restriction $\alpha|_{Ap}$ is \emph{inner}, meaning that there is a unitary $u \in Ap$ such that $\alpha(ap) = uapu^*$ for $a \in Ap$. The restriction $\alpha|_{A(1-p)}$ is said to be \emph{properly outer}.
    
    We are now ready to define quasi-innerness, proper outerness and residual proper outerness for a single *-automorphism. Kishimoto introduced the notion of a properly outer *-automorphism in \cite{Kishimoto_freely_acting} under the name ``freely acting,'' in terms of the Borchers spectrum of the action on invariant ideals. The definition of proper outerness that we give here differs from Kishimoto's definition, but it is equivalent by \cite[Section~7]{hamana85-injective_envelopes_equivariant}. We also note that this definition of proper outerness is equivalent to \emph{aperiodicity} as introduced by Kwa\'{s}niewski and Meyer (see e.g \cite{kwasniewski_meyer_aperiodicity}), which is defined more intrinsically in terms of hereditary subalgebras instead of the injective envelope. We are grateful to Bartosz Kwa\'{s}niewski for bringing this to our attention.
    
    \begin{definition}
    \label{def:residual_proper_outerness}
        Let $\alpha : A \to A$ be a *-automorphism of a C*-algebra $A$ and let $\tilde{\alpha} : I(A) \to I(A)$ denote the unique extension of $\alpha$ to a *-automorphism of $I(A)$. The *-automorphism $\alpha$ is said to be \emph{quasi-inner} if $\tilde{\alpha}$ is inner. It is said to be \emph{properly outer} if $\tilde{\alpha}$ is properly outer. It is said to be \emph{residually properly outer} if for every proper ideal $I \triangleleft A$ with $\alpha(I) = I$, the *-automorphism induced by $\alpha$ on the quotient $A/I$ is properly outer.
    \end{definition}

    %TODO: This isn't the original instance of "aperiodic". Actually tracking it down is annoying. It dates back at least to some work of Connes.
    %We note that proper outerness as defined above is equivalent to aperiodicity in the sense of \cite{kwasniewski_meyer_aperiodicity}.

    \begin{remark}
    We now make an important remark about terminology. If $\alpha : A \to A$ is a *-automorphism of a C*-algebra $A$, we will say that an ideal $I \triangleleft A$ is \emph{$\alpha$-invariant} if $\alpha(I) = I$. This is necessary for $\alpha$ to induce a *-automorphism on the quotient $A/I$.
    \end{remark}

   \begin{definition}
        An action $\alpha : \Gamma \curvearrowright A$ of a discrete group $\Gamma$ on a C*-algebra $A$ by *-automorphisms is \emph{pointwise residually properly outer} if for every $g \in G \setminus \{e\}$, the *-automorphism $\alpha_g$ is residually properly outer.
    \end{definition}

    \begin{remark}
        The above definition of pointwise residual proper outerness for a group action appears in the recent paper \cite[Section~3.3]{pitts_smith_zarikian_2023}, where it is called ``residual proper outerness.'' However, this name has already been used in the literature to refer to a related but different property (see e.g.\ \cite[Definition~7.8]{kirchberg_sierakowski_2016} and \cite[Definition~4.1]{kirchberg_sierakowski_2018}). We have chosen to use the terminology ``pointwise residual proper outerness'' to avoid confusion.
    \end{remark}
    
    The next result will be useful in Section \ref{sec:counterexample_constructions:noncommutative}, when we need to work with *-automorphisms that are not residually properly outer.
    
    \begin{proposition}
    \label{prop:not_residually_outer_iff_quasi_inner}
        Let $\alpha : A \to A$ be an automorphism of a C*-algebra $A$. If $\alpha$ is not residually properly outer, then there is a proper $\alpha$-invariant ideal $I \triangleleft A$ such that the action induced by $\alpha$ on $A/I$ is quasi-inner.
    \end{proposition}

    \begin{proof}
        The automorphism $\alpha$ not being residually properly outer means, by definition, that there is an $\alpha$-invariant ideal $J \triangleleft A$ such that the *-automorphism $\alpha_0 : A/J \to A/J$ induced by $\alpha$ on $A/J$ is not properly outer. Let $\tilde{\alpha}_0 : I(A/J) \to I(A/J)$ denote the unique extension of $\alpha_0$ to a *-automorphism of $I(A/J)$. Then there is a $\tilde{\alpha}_0$-invariant central projection $p$ such that the restriction $\tilde{\alpha}_0|_{I(A/J)p}$ is inner. 

        Let $\pi : A \to I(A/J)p$ denote the *-homomorphism obtained by composing the quotient *-homomorphism $A \to A/J$ with the compression *-homomorphism $A/J \to I(A/J)p$. Let $I = \ker \pi$. Then by Proposition~\ref{prop:injective_envelope_of_central_quotient},
        \[ I(A/I) = I((A/J)p) = I(A/J)p. \]
        Hence from above, the *-automorphism induced by $\alpha$ on $A/I$ is quasi-inner.
    \end{proof}
    
    It is natural to ask if the ideal $I$ in Proposition \ref{prop:not_residually_outer_iff_quasi_inner} can be chosen maximal or at least maximal with the property of being $\alpha$-invariant. However, this is not possible in general. This may seem surprising in light of the fact that if $\alpha$ is inner, then (1) every ideal of $A$ is $\alpha$-invariant and (2) the *-automorphism induced by $\alpha$ on a quotient by an $\alpha$-invariant ideal is inner. However, the next two examples demonstrate that neither of these properties holds in general for quasi-inner *-automorphisms.

    \begin{example}
    \label{ex:quasi_innerness_ideals_not_invariant}
        Let $\theta : X \to X$ be a free homeomorphism of a compact Hausdorff space $X$, and let $\alpha : \rC(X) \to \rC(X)$ denote the corresponding *-automorphism on the C*-algebra $\rC(X)$. Choose a faithful representation $\rC(X) \subseteq \B(H)$ such that $\alpha$ extends to an inner *-automorphism $\tilde{\alpha} : \B(H) \to \B(H)$ of $\B(H)$, i.e. such that there is a unitary $u \in \B(H)$ such that $\alpha(f) = ufu^*$ for all $f \in \rC(X)$. This can be done, for example, by identifying $\alpha$ with a $\mathbb{Z}$-action $\alpha : \Z \curvearrowright \rC(X)$, then choosing  a faithful representation of the crossed product $\rC(X) \rtimes_{\alpha,\lambda} \Z$. Let $A = \mathrm{K}(H) + \rC(X)$, where $\mathrm{K}(H)$ denotes the compact operators on $H$. Then $A$ is a C*-algebra. Let $\beta = \tilde{\alpha}|_A$. Then $\beta$ is a *-automorphism. Since $I(\mathrm{K}(H)) = \B(H)$, it follows that $I(A) = \B(H)$. In particular, $\beta$ is quasi-inner. However, it is clear that ideals of $A$ of the form $\mathrm{K}(H) + \rC_0(X \setminus \{x\})$ are not $\beta$-invariant.
    \end{example}

    \begin{example}
    \label{ex:quasi_innerness_doesnt_pass_quotients}
        Let $B$ be a simple C*-algebra, and let $\alpha : B \to B$ be an outer *-automorphism. Then $\alpha$ is actually properly outer (see e.g. \cite[Corollary~7.8]{hamana85-injective_envelopes_equivariant}). As in the previous example, we can choose a faithful representation $B \subseteq \B(H)$ such that $\alpha$ extends to an inner *-automorphism on $\B(H)$, and by restricting this extension to the C*-algebra $A = B + \mathrm{K}(H)$, we obtain a quasi-inner *-automorphism of $A$. However, $\mathrm{K}(H)$ is an $\alpha$-invariant ideal of $A$, and the *-automorphism induced by $\alpha$ on $A/\mathrm{K}(H) = B$ is not quasi-inner by the initial assumption on $B$. 
    \end{example}

    We will require the following result about injective envelopes in Section \ref{sec:counterexample_constructions}. 

    \begin{proposition}
    \label{prop:injective_envelope_of_central_quotient}
        Let $A$ be a C*-algebra and let $p \in I(A)$ be a central projection. Then $I(Ap) = I(A)p$.
    \end{proposition}

    \begin{proof}
        It is clear that $I(A)p$ is injective, so it suffices to prove that the inclusion $Ap \subseteq I(A)p$ is rigid. For this, let $\phi : I(A)p \to I(A)p$ be a unital completely positive map such that $\phi|_{Ap} = \id_{Ap}$. Then the unital completely positive map $\psi : I(A) \to I(A)$ defined by $\psi(a) = \phi(ap) + a(1-p)$ satisfies $\psi|_A = \id_A$, so by the rigidity of the inclusion $A \subseteq I(A)$, $\psi = \id_{I(A)}$. It follows that $\phi = \id_{I(A)p}$. 
    \end{proof} 

    \section{Sufficient condition for a Galois correspondence}
    \label{sec:free_action_correspondence}
    
    In this section we will prove Theorem~\ref{thmintro:main}, which asserts that if $\alpha : \Gamma \curvearrowright A$ is an action of a discrete group $\Gamma$ on a unital C*-algebra $A$ by *-automorphisms such that $\Gamma$ has the approximation property and either $\alpha$ is pointwise residually properlly outer or $\card{\Gamma} \leq 2$, then every intermediate C*-algebra for the inclusion $A \subseteq A \rtimes_{\alpha,\lambda} \Gamma$ is the reduced crossed product of a partial subaction of $\alpha$.

    As mentioned in the introduction, the key technical step in the proof is an averaging argument to show that if $B$ is an intermediate C*-algebra for the inclusion $A \subseteq A \rtimes_{\alpha,\lambda} \Gamma$ and $b \in B$ has Fourier series $b \sim \sum_{g \in \Gamma} b_g \lambda_g$, then $b_g \lambda_g \in B$ for all $g \in \Gamma$.

    We now give a rough sketch of this argument. First, using the pointwise residual proper outerness of $\alpha$, we will show that if the support of $b \in A \rtimes_{\alpha,\lambda} \Gamma$ contains at most one element of $\Gamma \setminus \{e\}$, then we can ``average'' $b$ arbitrarily close to the conditional expectation $E(b)$ using certain contractive $A$-bimodule operations. The contractivity means that this process can be iterated, and using this we will be able to show that we can average any finitely supported element $b \in A \rtimes_{\alpha,\lambda} \Gamma$ arbitrarily close to $E(b)$. It will follow immediately from the density of the finitely supported elements in $A \rtimes_{\alpha,\lambda} \Gamma$ that this can be done for arbitrary elements.
    
    Once we have shown that every $b \in A \rtimes_{\alpha,\lambda} \Gamma$ can be averaged arbitrarily close to $E(b)$, an easy translation argument will show that if $b$ has Fourier series $b \sim \sum_{g \in \Gamma} b_g \lambda_g$, then $b$ can also be averaged arbitrarily close to $b_g \lambda_g$ for all $g \in \Gamma$. If $B$ is an intermediate C*-algebra for the inclusion $A \subseteq A \rtimes_{\alpha,\lambda} \Gamma$, then $B$ is an $A$-bimodule, so it will follow from this that $b_g \lambda_g \in B$ for all $g \in \Gamma$.

    For the above averaging argument, we will utilize a particular kind of $A$-bimodule operation arising in the theory of C*-convexity.

    \begin{definition}
        Let $A \subseteq B$ be a unital inclusion C*-algebras. An $A$-bimodule $K \subseteq B$ is \emph{absolutely $A$-convex} if for every $n \geq 1$, every sequence $(x_i)_{i=1}^n$ in $K$ and every pair of sequences $(a_i)_{i=1}^n$ and $(b_i)_{i=1}^n$ in $A$ satisfying  $\sum_{i=1}^n a_i a_i^* \leq 1$ and $\sum_{i=1}^n b_i^* b_i \leq 1$, we have
        \[
            \sum_{i=1}^n a_i x_i b_i \in K
        \]
        We will refer to the above expression as an \emph{absolute $A$-convex combination}.
    \end{definition}

    First, we make the easy observation that an absolute $A$-convex combination of a sequence of contractions is a contraction. 

    \begin{lemma}
    \label{lem:absolutely_convex_combination_contractive}
        Let $A \subseteq B$ be a unital inclusion C*-algebras. If a sequence $(x_i)_{i=1}^n$ in $B$ satisfies $\|x_i\| \leq 1$ for all $i$, then any absolute $A$-convex combination $\sum_{i=1}^n a_i x_i b_i$ satisfies $\|\sum_{i=1}^n a_i x_i b_i\| \leq 1$.
    \end{lemma}

    \begin{proof}
        Let
        \[
        A = \begin{bmatrix} a_1 & \cdots & a_n \end{bmatrix},\quad B = \begin{bmatrix} b_1 \\ \vdots \\ b_n \end{bmatrix},\quad X = \begin{bmatrix} x_1 & & \\ & \ddots & \\ & & x_n \end{bmatrix}.
        \]
        Then the conditions $\sum_{i=1}^n a_i a_i^* \leq 1$, $\sum_{i=1}^n b_i^* b_i \leq 1$ and $\|x_i\| \leq 1$ for all $i$ are equivalent to $\|A\| \leq 1$, $\|B\| \leq 1$ and $\|X\| \leq 1$ respectively. Hence
        \[
        \left\| \sum_{i=1}^n a_i x_i b_i \right\| = \|A X B\| \leq \|A\|\|X\|\|B\| \leq 1. \qedhere
        \]
    \end{proof}

    We will prove a strong separation result for the absolutely $A$-convex sets that arise in our setting. For this, we will first require the following specialization of a separation theorem of Magajna \cite[Theorem~1.1(ii)]{magajna_bimodules}. Magajna's theorem applies to bimodules where the left and right C*-algebra are allowed to differ, which we do not require.

    \begin{proposition}
    \label{prop:magajna_separation}
        Let $A \subseteq B$ be a unital inclusion C*-algebras and let $K \subseteq B$ be a a norm-closed $A$-absolutely convex set. Suppose $y \in B \setminus K$. Then there is a Hilbert space $H$, a unital *-homomorphism $\pi : A \to \B(H)$ and a completely bounded $A$-bimodule map $\phi : B \to \B(H)$ with the property that $\|\phi(x)\| \leq 1$ for all $x \in K$ but $\|\phi(y)\| > 1$.
    \end{proposition}

    Note that the assertion $\phi$ is an $A$-bimodule map means that $\phi(axb) = \pi(a)\phi(x)\pi(b)$ for $a,b \in A$ and $x \in B$. 
    
    \begin{proof}
        Applying Magajna's two-sided separation theorem \cite[Theorem~1.1(ii)]{magajna_bimodules}, we obtain a Hilbert space $H_0$, unital *-homomorphisms $\pi_1 : A \to \B(H)$ and $\pi_2 : A \to \B(H_0)$, and a completely bounded $A-A$-module map $\phi_0 : B \to \B(H)$  with the property that $\|\phi_0(x)\| \leq 1$ for all $x \in K$ but $\|\phi(y)\| > 1$. Note that the assertion $\phi_0$ is an $A-A$ module map means that $\phi_0(a x b) = \pi_1(a) \phi_0(x) \pi_2(b)$ for $a,b \in A$ and $x \in B$.
        
        Let $H = H_0 \oplus H_0$ and define a new *-homomorphism $\pi : A \to \B(H)$ by
        \[ \pi(a) = \begin{bmatrix} \pi_1(a) & 0 \\ 0 & \pi_2(a) \end{bmatrix} \quad \text{for} \quad a \in A. \]
        Similarly, define a completely bounded map $\phi : B \to \B(H)$ by
        \[ \phi(b) = \begin{bmatrix} 0 & \Phi(b) \\ 0 & 0 \end{bmatrix} \quad \text{for} \quad b \in B. \]
        Then it is easy to verify that $\phi$ is an $A$-bimodule satisfying the desired properties.
    \end{proof}

    The $A$-convex sets that we will consider will be subsets of intermediate C*-algebras for the inclusion $A \subseteq A \rtimes_{\alpha,\lambda} \Gamma$, and the elements that we will be interested in will belong to $A$ itself. In this setting, we are able to obtain a significantly stronger separation result. 

    \begin{proposition}
    \label{prop:starhom_separation}
        Let $A \subseteq B$ be a unital inclusion of C*-algebras and let $K \subseteq B$ be a norm-closed absolutely $A$-convex set. Suppose $y \in A \cap (B \setminus K)$. Then there is a Hilbert space $H$, a unital *-homomorphism $\pi : A \to \B(H)$ and a completely bounded $A$-bimodule map $\phi : B \to \B(H)$ with $\phi|_A = \pi$ such that $\|\phi(x)\| \leq \xi$ for all $x \in K$ but $\|\phi(y)\| > \xi$ for some constant $\xi > 0$.
    \end{proposition}

    \begin{proof}
        Applying Proposition \ref{prop:magajna_separation}, we obtain a Hilbert space $H_0$, a unital *-homomorphism $\pi_0 : A \to \B(H)$ and a completely bounded $A$-bimodule map $\phi_0 : B \to \B(H)$ with the property that $\|\phi_0(x)\| \leq 1$ for all $x \in K$ but $\|\phi_0(y)\| > 1$. Let $c = \phi_0(1)$. The fact that $\phi_0$ is an $A$-bimodule map implies that $c \in \pi_0(A)'$. We will first show that we can assume $c \geq 0$. 
        
        Let $c = u|c|$ be the polar decomposition of $c$ and note that $u \in \pi_0(A)'$. Define $\phi_1 : B \to \B(H)$ by $\phi_1(b) = u^*\phi_0(b)$. Then $\phi_1$ is a completely bounded $A$-bimodule map. Furthermore, for $x \in K$,
        \[ \|\phi_1(x)\| = \|u^*\phi_0(x)\| \leq \|\phi_0(x)\| \leq 1, \]
        while the fact that $y \in A$ implies
        \[ \|\phi_1(y)\| = \|u^* c \pi_0(y)\| = \|\pi_0(y)^* c^*c \pi_0(y)\|^{1/2} = \|c \pi_0(y)\| = \|\phi_0(y)\| > 1. \]
        Therefore, by replacing $\phi_0$ by $\phi_1$, we can assume that $c \geq 0$.

        Let $Z$ denote the spectrum of the commutative C*-algebra $\ca(z,1)$ and identify $\ca(z,1)$ with $\rC(Z)$. Then letting $C = \ca(\pi_0(A),z)$, $\rC(Z) \subseteq \rZ(C) \subseteq C$. For $z \in Z$, let $I_z \triangleleft C$ denote the ideal in $C$ generated by the corresponding maximal ideal in $\rC(Z)$. By the Cohen-Hewitt factorization theorem, $I_z = C \cdot \rC_0(Z \setminus \{z\})$. Let $p_z \in C^{**}$ denote the central support projection for $I_z$, and note that we may view $C^{**} \subseteq \B(H)^{**}$. Let $\sigma_z : \B(H) \to p_z \B(H)^{**} p_z$ denote compression by $p_z$. Note that the restriction $\sigma_z|_C$ is a *-homomorphism from $C$ to $C^{**}p_z$ with $\ker \sigma_z|_C = I_z$.

        The norm of $\phi_0(y)$ can be computed as
        \[ \|\phi_0(y)\| = \|\pi_0(y) c \| = \sup_{z \in Z} \|\sigma_z(\pi_0(y)) c(z)\|. \]
        Since $\|\phi_0(y)\| > 1$, it follows that $\|\sigma_{z_0}(\pi_0(y)) c(z_0)\| > 1$ for some $z_0 \in Z$. Note in particular that $c(z_0) > 0$.

        We now define the completely bounded map $\phi : B \to p_{z_0} \B(H)^{**} p_{z_0}$ by
        \[ \phi(b) = c(z_0)^{-1} \sigma_{z_0}(\phi_0(b)) \quad \text{for} \quad b \in B. \]
        Letting $\xi = c(z_0)^{-1}$, it is easy to see that $\|\phi(x)\| \leq \xi$ for all $x \in K$, while $\|\phi(y)\| > \xi$. Furthermore, for $a \in A$,
        \[
            \phi(a) = c(z_0)^{-1} \sigma_{z_0}(\phi_0(a)) = c(z_0)^{-1} \sigma_{z_0}(\pi_0(a)c) = \sigma_{z_0}(\pi_0(a)),
        \]
        so $\phi|_A$ is a *-homomorphism and we may define $\pi = \phi|_A$. This also lets us show that for $a \in A$ and $b \in B$,
        \begin{align*}
            \phi(a) \phi(b) &= \sigma_{z_0}(\pi_0(a)) \cdot c(z_0)^{-1} \sigma_{z_0}(\phi_0(b)) = c(z_0)^{-1} \sigma_{z_0}(\pi_0(a) \phi_0(b)) \\
            &= c(z_0)^{-1} \sigma_{z_0}(\phi_0(ab)) = \phi(ab),
        \end{align*}
        where we have used the fact that $\pi_0(a)$ lies in the multiplicative domain of the completely positive map $\sigma_{z_0}$, which restricts to a *-homomorphism on $C$ from above. Similarly, $\phi(b) \phi(a) = \phi(ba)$, and we conclude that $\phi$ is an $A$-bimodule map.
    \end{proof}

    We now apply our separation result, Proposition \ref{prop:starhom_separation}, to show that if $b \in A \rtimes_{\alpha,\lambda} \Gamma$ is supported on at most two elements of $\Gamma$, then we can ``average'' $b$ arbitrarily close to the conditional expectation $E(b)$ using absolute $A$-convex combinations.

    \begin{lemma}
    \label{lem:average_one_element_to_expectation}
        Let $\alpha : \Gamma \curvearrowright A$ be an action of a discrete group $\Gamma$ on a unital C*-algebra $A$. For $g \in G \setminus \{e\}$, suppose that $\alpha_g$ is residually properly outer. Let $b \in A \rtimes_{\alpha,\lambda} \Gamma$ be an element of the form $b = b_e + b_g \lambda_g$ for $b_e,b_g \in A$. Then $b_e = E(b)$ belongs to the norm closure of the absolute $A$-convex hull of $b$.
    \end{lemma}

    \begin{proof}
        Let $K \subseteq A \rtimes_{\alpha,\lambda} \Gamma$ denote the norm closure of the absolutely $A$-convex hull of $b$ and suppose for the sake of contradiction that $b_e = E(b) \notin K$. We will construct an $\alpha_g$-invariant ideal $J$ of $A$ such that the induced action of $\alpha_g$ on $A/J$ is not properly outer, contradicting the assumption that $\alpha_g$ is residually properly outer.
        
        Applying Proposition \ref{prop:starhom_separation}, there is a Hilbert space $H$, a unital *-homomorphism $\pi : A \to \B(H)$ and a completely bounded $A$-bimodule map $\phi : A \rtimes_{\alpha,\lambda} \Gamma \to \B(H)$ with $\phi|_A = \pi$ such that $\norm{\phi(x)} \leq \xi$ for all $x \in K$ but $\|\phi(b_e)\| > \xi$ for some constant $\xi > 0$.

        Let $I = \ker \pi$. By the injectivity of the injective envelope of the quotient $I(A/I)$, there is a unital completely positive map $\psi : \B(H) \to I(A/I)$ extending the *-isomorphism from $\pi(A)$ to $A/I$. Note that $\pi(A)$ belongs to the multiplicative domain of $\psi$. 

        Let $\theta = \psi \circ \phi$. Then $\|\theta(x)\| \leq \|\phi(x)\| \leq \xi$ for $x \in K$. Also, $\theta|_A = \psi \circ \pi = \sigma$, where $\sigma : A \to A/I$ is the quotient *-homomorphism. Since $b_e \in A$, this implies
        \[ \|\theta(b_e)\| = \|\sigma(b_e)\| = \|\pi(b_e)\| = \|\phi(b_e)\| > \xi. \]
        Furthermore, since $\phi$ is an $A$-bimodule map and $\pi(A)$ belongs to the multiplicative domain of $\psi$, $\theta$ is also an $A$-bimodule map, i.e. $\theta(a_1 c a_2) = \sigma(a_1) \theta(c) \sigma(a_2)$ for $a_1,a_2 \in A$ and $c \in A \rtimes_{\alpha,\lambda} \Gamma$. 
        
        Now, since $b \in K$,
        \[ \|\sigma(b_e) + \sigma(b_g) \theta(\lambda_g) \| = \|\theta(b)\| \leq \xi < \|\theta(b_e)\|, \]
        which implies that $\theta(\lambda_g) \ne 0$. Let $r = \theta(\lambda_g)$. Observe that for $a \in A$,
        \[ r \sigma(a) = \theta(\lambda_g a) = \theta(\alpha_g(a) \lambda_g) = \sigma(\alpha_g(a)) r. \]
        Similarly,
        \[ r^* \sigma(a) = (\sigma(a^*) r)^* = \theta(a^* \lambda_g)^* = \theta(\lambda_g \alpha_{g^{-1}}(a^*))^* = (r \sigma(\alpha_{g^{-1}}(a^*)))^* = \sigma(\alpha_{g^{-1}}(a)) r^*. \]
        Hence
        \[ r^*r \sigma(a) = r^* \sigma(\alpha_g(a)) r = \theta(a) r^*r, \]
        so that $r^*r$ commutes with $A/I$ in $I(A/I)$. Hence by \cite[Corollary~4.3]{hamana79_injective_envelopes_cstaralg}, $r^*r$ belongs to the center $Z(I(A/I))$ of $I(A/I)$. Similarly, $rr^*$ belongs to the center of $I(A/I)$.

        Let $r = u|r|$ be the polar decomposition of $r$. Then $p = u^*u$ is the left support projection for $r$ and $|r^*|$. Similarly, $q = uu^*$ is the right support projection for $r$ and $|r|$. Since $|r|,|r^*| \in Z(I(A/I))$, it follows that $p,q \in Z(I(A/I))$ and hence that $p = q$. Note that this also implies that $p$ is the (left and right) support projection for $u$ and $u^*$. It follows from above that $u$ is a unitary in $I(A/I)p$ satisfying $u \sigma(a) = \sigma(\alpha_g(a)) u$ for $a \in A$.

        We would like to conclude that the induced action of $\alpha_g$ on $A/I$ is not properly outer, which would contradict the assumption that $\alpha_g$ is residually properly outer. However, the ideal $I$ is not necessarily $\alpha_g$-invariant. Instead, define $\rho : A \to I(A/I)p$ by $\rho(a) = \sigma(a)p$. Then $\rho$ is a *-homomorphism satisfying $\rho(\alpha_g(a)) = u \rho(a) u^*$ for all $a \in A$. Hence $J = \ker \rho$ is $\alpha_g$-invariant, and the induced action of $\alpha_g$ on $A/J$ is quasi-inner, as by Proposition~\ref{prop:injective_envelope_of_central_quotient} we have $I(A/J) \isoto I(A/I)p$.
    \end{proof}

    With the previous lemma in hand, we have now overcome the main technical hurdle. Before proceeding further, we make an important remark.

    \begin{remark} \label{rem:composition_convex_combinations}
        It is a slightly subtle point that the absolute $A$-convex hull of a single element $x$ can be characterized as the norm closure of
        \[ \setbuilder{\sum_{i=1}^n a_i x b_i}{a_1,\ldots,a_n,b_1,\ldots,b_n \in A \text{ with } \sum_{i=1}^n a_i a_i^* \leq 1,\ \sum_{i=1}^n b_i^*b_i \leq 1 \text{ for } n \geq 1}, \]
        simply because absolute $A$-convex combination of absolute $A$-convex combinations can always be written as a single absolute $A$-convex combination.
    \end{remark}

    In the next result, we show that the averaging process in Lemma \ref{lem:average_one_element_to_expectation} can be applied to arbitrary elements by iterating. 

    \begin{proposition}
    \label{prop:average_all_to_expectation}
        Let $\alpha : \Gamma \curvearrowright A$ be a pointwise residually properly outer action of a discrete group $\Gamma$ on a unital C*-algebra $A$. Then for every element $b \in A \rtimes_{\alpha,\lambda} \Gamma$, the conditional expectation $E(b)$ belongs to the norm closure of the absolute $A$-convex hull of $b$.
    \end{proposition}

    \begin{proof}
        First suppose that $b \in A \rtimes_{\alpha,\lambda} \Gamma$ is finitely supported, say
        \[ b = b_e + b_{g_1} \lambda_{g_1} + \cdots + b_{g_n} \lambda_{g_n} \]
        for distinct $g_1,\ldots,g_n \in \Gamma$. Fix $\varepsilon > 0$. By Lemma \ref{lem:average_one_element_to_expectation}, there is an absolute $A$-convex combination $\psi_1(\cdot) = \sum_i a_i \cdot b_i$ such that $\|\psi_1(b_e + b_{g_1} \lambda_{g_1}) - b_e\| < \varepsilon/n$. Therefore, we can write
        \[ \psi_1(b) = b_e + b_{g_2}' \lambda_{g_2} + \cdots + b_{g_n}' \lambda_{g_n} + r_1, \]
        for $b_{g_2}',\ldots,b_{g_n}' \in A$ and $\|r_1\| < \varepsilon/n$. 

        Applying Lemma \ref{lem:average_one_element_to_expectation} again, there is an absolute $A$-convex combination $\psi_2$ such that $\|\psi_2(b_e + b_{g_2}' \lambda_{g_2}) - b_e\| < \varepsilon/n$. Therefore, we can write
        \[
            \psi_2 \circ \psi_1(b) = b_e + b_{g_3}'' \lambda_{g_3} + \cdots + b_{g_n}'' \lambda_{g_n} + \psi_2(r_1) + r_2,
        \]
        for $b_{g_3}'',\ldots,b_{g_n}'' \in A$ and $r_2 < \varepsilon/n$. Note that $\|\psi_2(r_1)\| < \varepsilon/n$ by Lemma \ref{lem:absolutely_convex_combination_contractive}, so $\|\psi_2(r_1) + r_2\|< 2 \varepsilon/n$. 

        Iterating this process $n-2$ more times by composing suitable $A$-convex combinations and applying Remark~\ref{rem:composition_convex_combinations} yields a single absolute $A$-convex combination $\psi$ such that $\psi(b) = b_e + r$ for $\|r\| < \varepsilon$. In particular, $\|\psi(b) - b_e\| < \varepsilon$.

        Now for an arbitrary element $b \in A \rtimes_{\alpha,\lambda} \Gamma$ and $\epsilon > 0$, we can choose a finitely supported element $b_0 \in A \rtimes_{\alpha,\lambda} \Gamma$ with $E(b_0) = E(b)$ and $\|b - b_0\| < \varepsilon/2$. From above there is an absolute $A$-convex combination $\psi$ such that $\|\psi(b_0) - E(b)\| < \varepsilon/2$. Hence
        \[ \|\psi(b) - E(b)\| \leq \|\psi(b - b_0)\| + \|\psi(b_0) - E(b)\| < \varepsilon, \]
        where we have applied Lemma \ref{lem:absolutely_convex_combination_contractive}. It follows that $E(b)$ belongs to the norm closure of the absolute $A$-convex hull of $b$.
    \end{proof}

    The following corollary is the desired averaging result. The proof is an application of a standard translation trick to Proposition \ref{prop:average_all_to_expectation},

    \begin{corollary}
        \label{cor:average_all_to_any_coefficient}
        Let $\alpha : \Gamma \curvearrowright A$ be a pointwise residually properly outer action of a discrete group $\Gamma$ on a unital C*-algebra $A$. Then for every element $b \in A \rtimes_{\alpha,\lambda} \Gamma$ with Fourier series $b \sim \sum_{g \in \Gamma} b_g \lambda_g$, the element $b_g \lambda_g$ belongs to the norm closure of the absolute $A$-convex hull of $b$ for all $g \in \Gamma$.
    \end{corollary}

    \begin{proof}
        Fix $g \in \Gamma$ and consider the element $b \lambda_g^*$, which satisfies $E(b \lambda_g^*) = b_g$. For $\varepsilon > 0$, applying Proposition \ref{prop:average_all_to_expectation} yields an absolute $A$-convex combination $\psi$ such that $\|\psi(b \lambda_g^*) - b_g\| < \epsilon$. Expand $\psi(\cdot) = \sum_i a_i^* \cdot b_i$ and define a new absolute $A$-convex combination $\psi_g(\cdot) = \sum_i a_i^* \cdot \alpha_{g^{-1}}(b_i)$. Then $\psi_g(b) = \psi(b \lambda_g^*) \lambda_g$, so
        \[ \|\psi_g(b) - b_g \lambda_g\| = \|\psi(b \lambda_g^*) \lambda_g - b_g \lambda_g\| \leq \|\psi(b \lambda_g^*) - b_g\| < \epsilon. \]
        It follows that $b_g \lambda_g$ belongs to the norm closure of the absolute $A$-convex hull of $b$.
    \end{proof}

    We are now ready to prove Theorem~\ref{thmintro:main}.

    \begin{proof}[Proof of Theorem~\ref{thmintro:main}]
        Assuming that $\Gamma$ has the approximation property, we want to show that if either $|\Gamma| \leq 2$ or the action $\alpha : \Gamma \curvearrowright A$ is pointwise residually properly outer, then every intermediate C*-algebra for the inclusion $A \subseteq A \rtimes_{\alpha,\lambda} \Gamma$ is the reduced crossed product of a partial subaction of $\alpha$. Fix an intermediate C*-algebra $B$ for this inclusion. Note that $B$ is an $A$-bimodule, and in particular is closed under absolute $A$-convex combinations.

        We first claim that for any element $b \in B$ with Fourier series $b \sim \sum_{g \in \Gamma} b_g \lambda_g$, we have $b_g \lambda_g \in B$ for all $g \in \Gamma$. The claim is trivial if $\Gamma$ is trivial. If $|\Gamma| = 2$, then the claim follows immediately from the fact that $A \subseteq B$. Otherwise, if $|\Gamma| > 2$, then $\alpha$ is pointwise residually properly outer by assumption, so the claim is implied by Corollary \ref{cor:average_all_to_any_coefficient}.

        We now construct the partial subaction $\beta : \Gamma \curvearrowright A$. For $g \in \Gamma$, let $I_g = \{a \in A : a \lambda_g \in B\}$. It is clear that $I_e = A$, and since $B$ is an $A$-bimodule, it follows easily from the first claim that each $I_g$ is a (closed) ideal of $A$.
        
        We claim that the restriction $\alpha_g|_{I_{g^{-1}}}$ is a *-isomorphism from $I_{g^{-1}}$ to $I_g$. To see this, first note that for $a \in I_{g^{-1}}$, the self-adjointness of ideals implies $a^* \in I_{g^{-1}}$, so $a^* \lambda_{g^{-1}} \in B$. Taking the adjoint gives $\lambda_g a = \alpha_g(a) \lambda_g \in B$, so $\alpha_g(a) \in I_g$. Hence $\alpha_g(I_{g^{-1}}) \subseteq I_g$. Similarly, $\alpha_{g^{-1}}(I_g) \subseteq I_{g^{-1}}$, and applying $\alpha_g$ to each side gives $I_g \subseteq \alpha_g(I_{g^{-1}})$. Therefore, $I_g = \alpha_g(I_{g^{-1}})$, and we see that $\alpha_g|_{I_{g^{-1}}}$ is a *-isomorphism from $I_{g^{-1}}$ to $I_g$ with inverse $\alpha_{g^{-1}}|_{I_g}$, proving the claim. 
        
        For $g \in \Gamma$, let $\beta_g = \alpha_g|_{I_{g^{-1}}}$. In order to conclude that $\beta$ is a partial subaction, we must verify that $\beta_h^{-1}(I_h \cap I_{g^{-1}}) \subseteq I_{(gh)^{-1}}$ for all $g,h \in \Gamma$. For this, we first note that it follows from basic C*-algebra theory that $I_h \cap I_{g^{-1}} = I_h I_{g^{-1}}$. Hence for $a \in I_h$ and $b \in I_{g^{-1}}$, $ab \in I_h$, so from above $\beta_{h^{-1}}(ab) \in I_{h^{-1}}$. Hence $\beta_{h^{-1}}(ab) \lambda_{h^{-1}} \in B$. We want to show that $\beta_{h^{-1}}(ab) \lambda_{(gh)^{-1}} = \beta_{h^{-1}}(ab) \lambda_{h^{-1}} \lambda_{g^{-1}} \in B$. For this, let $(e_i)$ be an approximate unit for $I_h \cap I_{g^{-1}}$. Then $e_i \lambda_{g^{-1}} \in B$ for each $i$, so $\beta_{h^{-1}}(ab) \lambda_{h^{-1}} e_i \lambda_{g^{-1}} \in B$. Now
        \[
            \beta_{h^{-1}}(ab) \lambda_{h^{-1}} e_i \lambda_{g^{-1}} = \beta_{h^{-1}}(ab e_i) \lambda_{h^{-1}} \lambda_{g^{-1}} \in B,
        \]
        and taking the limit implies $\beta_{h^{-1}}(ab) \lambda_{h^{-1}} \lambda_{g^{-1}} \in B$ as required. Hence $\beta$ is a partial subaction of $\alpha$.

        From above, the closed span of $\{a_g \lambda_g : a_g \in I_g,\ g \in \Gamma\}$ is contained in $B$. The reverse inclusion follows from the first claim and Proposition \ref{prop:ap_crossed_products}.
    \end{proof}

    In the special case when $A$ is simple, the action $\alpha : A \curvearrowright \Gamma$ is pointwise residually properly outer if and only if it is properly outer. Moreover, it is known that for simple C*-algebras, proper outerness is equivalent to outerness (see e.g. \cite[Corollary~7.8]{hamana85-injective_envelopes_equivariant}). Therefore, the main result in \cite{cameron_smith_cstar_simple} is equivalent to the assertion that if $A$ is simple, then Theorem~\ref{thmintro:main} holds even if $\Gamma$ does not have the approximation property. For completeness, we now indicate how this result can be proved by a simple modification of our arguments.

    \begin{theorem}
         Let $\alpha : \Gamma \curvearrowright A$ be an action of a discrete group $\Gamma$ on a simple unital C*-algebra $A$ by *-automorphisms. If either $\card{\Gamma} \leq 2$ or $\alpha$ is pointwise outer, then every intermediate C*-algebra for the inclusion $A \subseteq A \rtimes_{\alpha,\lambda} \Gamma$ is the reduced crossed product $A \rtimes_{\alpha|_\Lambda,\lambda} \Lambda$ of a restricted action $\alpha|_\Lambda$ for a subgroup $\Lambda \leq \Gamma$.
    \end{theorem}

    \begin{proof}
        Let $B$ be an intermediate C*-algebra. As mentioned above, outerness is equivalent to proper outerness for simple C*-algebras, so the action $\alpha$ is pointwise residually properly outer. Arguing as in the proof of Theorem~\ref{thmintro:main}, given any element $b \in B$ with Fourier series $b \sim \sum_{g \in \Gamma} b_g \lambda_g$, we have $b_g \lambda_g \in B$ for all $g \in \Gamma$. We also obtain an ideal $I_g = \{a \in A : a \lambda_g \in B \}$ of $A$ for all $g \in \Gamma$, and by the simplicity of $A$, each $I_g$ is either $0$ or $A$. 

        Let $\Lambda = \{g \in \Gamma : I_g = A \}$. It is easy to see that $\Lambda$ is a subgroup of $\Gamma$, since $g \in \Lambda$ if and only if $\lambda_g \in B$. Therefore, $A \rtimes_{\alpha|_\Lambda,\lambda} \Lambda \subseteq B$. 

        For the reverse inclusion, let $E_{\Lambda} : A \rtimes_{\alpha,\lambda} \Gamma \to A \rtimes_{\alpha|_\Lambda,\lambda} \Lambda$ denote the conditional expectation satisfying
        \[ E_{\Lambda}(a \lambda_g) = \begin{cases} a \lambda_g & g \in \Lambda \\ 0 & \text{otherwise} \end{cases}, \quad \text{for} \quad a \in A,\ g \in \Gamma. \]
        Since $E_{\Lambda}$ is the identity on $\ca_\lambda(\Lambda)$, $\ca_\lambda(\Lambda)$ belongs to the multiplicative domain of $E_{\Lambda}$. Hence for $b \in A \rtimes_{\alpha,\lambda} \Gamma$ and $g \in \Lambda$, $E(E_{\Lambda}(b)\lambda_g^*) = E(E_{\Lambda}(b \lambda_g^*)) = E(b \lambda_g^*)$. In other words, $E_{\Lambda}$ preserves the Fourier coefficients corresponding to elements of $\Lambda$. Hence for $b \in B$, $E_{\Lambda}(b)$ has the same Fourier coefficients as $b$, implying $b = E_{\Lambda}(b) \in A \rtimes_{\alpha|_\Lambda,\lambda} \Lambda$. Therefore, $B \subseteq A \rtimes_{\alpha|_\Lambda,\lambda} \Lambda$.
    \end{proof}
    
    \section{Necessary conditions for a Galois correspondence}
    \label{sec:counterexample_constructions}

    \subsection{The noncommutative case}
    \label{sec:counterexample_constructions:noncommutative}

    In this section we will prove Theorem~\ref{thmintro:converse}, establishing as much as possible the converse to Theorem~\ref{thmintro:main}. Specifically, let $\alpha : \Gamma \curvearrowright A$ be the action of a discrete group on a unital C*-algebra $A$ by *-automorphisms. If $\alpha$ is not pointwise residually properly outer, then our goal is to explicitly construct an intermediate C*-algebra $A \subseteq B \subseteq A \rtimes_{\alpha,\lambda} \Gamma$ such that $B$ is not the reduced crossed product of a partial subaction of $\alpha$. Note that we will not require the approximation property in this section.

    The following definition will be quite useful for us. The terminology is inspired by the notion of a ``rigid stabilizer'' in topological dynamics.

    \begin{definition}
    \label{def:nc_rigid_stabilizer}
        Let $\alpha : \Gamma \curvearrowright A$ be an action of a discrete group $\Gamma$ on a C*-algebra $A$ by *-automorphisms. For an ideal $K \triangleleft A$, the \emph{rigid stabilizer} of $K$ is the subgroup $\Gamma_K \leq K$ consisting of elements $g \in \Gamma$ such that $K$ is $\alpha_g$-invariant and the *-automorphism induced by $\alpha_g$ on the quotient $A/K$ is quasi-inner.
    \end{definition}

    \begin{remark}
    \label{rem:rigid_stabilizer}
    Proposition~\ref{prop:not_residually_outer_iff_quasi_inner} implies that if the action $\alpha : \Gamma \curvearrowright A$ is not pointwise residually properly outer, then there is a proper ideal $K \triangleleft A$ such that the corresponding rigid stabilizer $\Gamma_K$ is non-trivial.
    \end{remark}

    The following result is straightforward, but it provides two important sources of intermediate C*-algebras that are not reduced crossed products of partial subactions. 

    \begin{lemma}
    \label{lem:intermediate_subalgebra_reductions}
        Let $\alpha : \Gamma \curvearrowright A$ be an action of a discrete group on a unital C*-algebra $A$ by *-automorphisms.
        \begin{enumerate}
            \item Let $\Lambda \leq \Gamma$ be a subgroup such that some intermediate C*-algebra $B$ for the inclusion $A \subseteq B \subseteq A \rtimes_{\alpha|_\Lambda, \lambda} \Lambda$ is not the reduced crossed product of a partial subaction of $\alpha|_{\Lambda}$. Then $B$ is also an intermediate C*-algebra for the inclusion $A \subseteq A \rtimes_{\alpha,\lambda} \Gamma$ that is not the reduced crossed product of a partial subaction of $\alpha$. 

            \item Let $J \triangleleft A$ be a $\Gamma$-invariant ideal such that some intermediate C*-algebra $B$ for the inclusion $A/J \subseteq A/J \rtimes_{\alpha_{A/J},\lambda} \Gamma$ is not the reduced crossed product of a partial subaction of $\alpha|_{A/J}$, where $\alpha|_{A/J}$ denotes the action induced by $\alpha$ on $A/J$. Then letting $\pi : A \rtimes_{\alpha,\lambda} \Gamma \to (A/J) \rtimes_{\alpha_{A/J},\lambda} \Gamma$ be the quotient *-homomorphism, the pullback $\pi^{-1}(B)$ is an intermediate C*-algebra for the inclusion $A \subseteq A \rtimes_{\alpha,\lambda} \Gamma$ that is not the reduced crossed product of a partial subaction of $\alpha$.
        \end{enumerate}
    \end{lemma}

    \begin{proof}
        The first statement is immediately clear. For the second statement, it suffices to observe that if $\beta : \Gamma \curvearrowright A$ is a partial subaction of $\alpha$, then the image of the reduced crossed product $A \rtimes_{\beta,\lambda} \Gamma$ under $\pi$ is the reduced crossed product $A/J \rtimes_{\beta_{A/J}, \lambda} \Gamma$ of the partial action $\beta|_{A/J}$ induced by $\beta$ on the quotient $A/J$.
    \end{proof}
    
    We are now ready to commence with the proof of Theorem \ref{thmintro:converse}. Fix an action $\alpha : \Gamma \curvearrowright A$ of a discrete group $\Gamma$ on a unital C*-algebra $A$ by *-automorphisms such that $\alpha$ is not pointwise residually properly outer. By Remark \ref{rem:rigid_stabilizer}, there is a proper ideal $K \triangleleft A$ such that the corresponding rigid stabilizer $\Gamma_K$ satisfies $\card{\Gamma_K} \geq 3$. We will divide the proof of Theorem \ref{thmintro:converse} into two cases, depending on the order of the elements in $\Gamma_K$ as follows: 
    \begin{enumerate}
        \item The first case is when $\Gamma_K$ contains an element of order at least $3$, with infinite order allowed. This is the easiest case and is done in Proposition~\ref{prop:counterexample_order_three}.

        \item The second, more difficult, case is when every non-trivial element in $\Gamma_K$ is of order $2$. This case is more difficult, and is done in Proposition~\ref{prop:counterexample_klein_four}.
    \end{enumerate}
    
    We begin with the easiest case when $\Gamma_K$ contains an element of order at least $3$.

    \begin{proposition}
    \label{prop:counterexample_order_three}
        If $\Gamma_K$ contains an element of order at least $3$, then there is an intermediate C*-algebra for the inclusion $A \subseteq A \rtimes_{\alpha,\lambda} \Gamma$ that is not the reduced crossed product of a partial subaction of $\alpha$.
    \end{proposition}

    \begin{proof}
        Let $t \in \Gamma_K$ be an element of order at least $3$. In light of Lemma~\ref{lem:intermediate_subalgebra_reductions}, we may assume without loss of generality that $\Gamma = \langle t \rangle$ and, by passing to $A/K$ and applying Lemma \ref{lem:intermediate_subalgebra_reductions}, that the action of $\Gamma$ on $A$ itself is quasi-inner. Recall that this means the unique extension of $\alpha_t$ to a *-automorphism of $I(A)$, which we continue to denote by $\alpha_t$ for convenience, is inner. This means there  is a unitary $u \in I(A)$ that implements $\alpha_t$, i.e. such that $\alpha_t = \Ad_u$.

        If $t$ is of infinite order, so that $\Gamma \isoto \Z$, then the pair $(\id_{I(A)},u)$ is a covariant representation for $\alpha$. Otherwise, if $t$ has finite order $n$, so that $\Gamma \isoto \Z/n\Z$, then the fact that $\alpha_t^n = \alpha_{t^n} = \alpha_e = \id_{I(A)}$ implies $u^n = z$ for a central unitary $z \in Z(I(A))$. Since $Z(I(A))$ is monotone complete, we can choose an $n$-th root, say $w \in Z(I(A))$ for $u$. Replacing $u$ by $uw^*$, we can assume that $u^n = 1$, so that the pair $(\id_{I(A)},u)$ is a covariant representation for $\alpha$.

        Since $\Gamma$ is amenable, the covariant representation   $(\id_{I(A)},u)$ corresponds to a *-homomorphism $\pi : I(A) \rtimes_{\alpha,\lambda} \Gamma \to I(A)$. Let $K = \ker \pi$. Clearly $K$ is non-zero. Also, the inclusion $A \rtimes_{\alpha,\lambda} \Gamma \subseteq I(A) \rtimes_{\alpha,\lambda} \Gamma$ is essential by \cite[Theorem 3.4]{hamana85-injective_envelopes_equivariant}. Therefore, $J = K \cap A \rtimes_{\alpha,\lambda} \Gamma$ is a non-zero ideal in $A \rtimes_{\alpha,\lambda} \Gamma$.

        The C*-algebra $B = A + J$ is an intermediate C*-algebra for the inclusion $A \subseteq A \rtimes_{\alpha,\lambda} \Gamma$. If $B$ is not the reduced crossed product of a partial subaction of $\alpha$, then we are done. So we may assume that $B$ is the reduced crossed product of a partial subaction $\beta : \Gamma \curvearrowright A$ of $\Gamma$ on $A$, say $B = A \rtimes_{\beta,\lambda} \Gamma$.

        It will be convenient to write elements in $\Gamma$ as integers. Let $\{I_k\}_{k \in \Gamma}$ denote the family of ideals of $A$ corresponding to $\beta$. Since $J \ne 0$, $I_k \ne 0$ for some $k$, and since $J$ is an ideal of $A \rtimes_{\alpha,\lambda} \Gamma$, meaning in particular that is closed under left and right multiplication by $\lambda_1$ and $\lambda_{-1}$, it follows that $I_1 \ne 0$. Note that since $\Gamma$ is abelian, $I_1$ is $\Gamma$-invariant. 

        The fact that $I_1 \lambda_1 \subseteq B = A + J$ implies that for $a \in I_1$, $a\lambda_1 = b + c$ for $b \in A$ and $c \in J$. Applying $\pi$ to this equality gives $a u = b$, so in particular $au \in A$. For $a \geq 0$, this gives $a u = a^{1/2} (a^{1/2}u) \subseteq I_1$. Hence $I_1 u \subseteq I_1$. Also, since $I_1$ is $\Gamma$-invariant, $I_1 \lambda_1 = \lambda_1 I_1$, so arguing similarly gives $u I_1 \subseteq I_1$.

        Let $p \in I(A)$ denote the central support projection for $I_1$, which is obtained as the supremum of an approximate unit for $I_1$. Then $v = up \in M(I_1)$ is a unitary in $M(I_1)$ (see e.g. \cite[Theorem 8.1.29]{SWMonotoneComplete}) that implements $\alpha_1$ on $I_1$. In other words, the restriction $\alpha_1|_{I_1}$ is inner. 

        In summary, we have found a nontrivial ideal $I \triangleleft A$ such that the restriction $\alpha_1|_I$ is inner, i.e. $\alpha_1|_I = \Ad_v$ with $v \in M(I)$. Further, if $\Gamma = \Z/n\Z$, then $v^n = 1$. This implies that
        \[ I \rtimes_{\alpha,\lambda} \Gamma \isoto I \tensor C^*_\lambda(\Gamma) \]
        via the *-isomorphism sending $a \tensor \lambda_k$ to $a v^k \tensor \lambda_k$ for $a \in I$ and $k \in \Gamma$. We will now utilize this *-isomorphism to construct intermediate C*-algebras that are not reduced crossed products of partial subactions of $\alpha$.

        First consider the case when $\Gamma = \Z/n\Z$. In this case, the sum
        \[ p = \frac{1}{n} \sum_{k \in \Gamma} \lambda_k \in C^*_\lambda(\Gamma) \]
        is a nontrivial rank one projection. Then $I \tensor p$ is the ideal in $I \tensor C^*_\lambda(\Gamma)$ consisting of elements with constant Fourier coefficients. In particular, the Fourier coefficients of a nonzero element in this ideal are all nonzero. Pulling back to $I \rtimes_{\alpha,\lambda} \Gamma$, we see that that an element of the form $\sum a_k \tensor \lambda_k \in I \otimes p$ corresponds to $\sum av^{-k} \lambda_k \in I \rtimes_{\alpha,\lambda} \Gamma$. Let $K \triangleleft I \rtimes_{\alpha,\lambda} \Gamma$ be the image of $I \otimes p \triangleleft I \otimes \ca_\lambda(\Gamma)$ under this pullback. From above, if an element in $K$ is nonzero, then all of its Fourier coefficients are nonzero. 

        Now $K$ is also an ideal in $A \rtimes_{\alpha,\lambda} \Gamma$, so $A+K$ is an intermediate C*-algebra for the inclusion $A \subseteq A \rtimes_{\alpha,\lambda} \Gamma$. Moreover, since $\card{\Gamma} \geq 3$, it follows from above that every element $b \in A+K \setminus A$ always has at least two nonzero Fourier coefficients. Therefore, $A +K$ is not the reduced crossed product of a partial subaction of $\alpha$.

        The argument for the case when $\Gamma = \Z$ is similar. Since $\ca_\lambda(\Gamma) \isoto \rC(\T)$,
        \[ I \rtimes_{\alpha,\lambda} \Gamma \isoto I \otimes \ca_\lambda(\Gamma) \isoto I \otimes \rC(\T). \]
        Pick a nontrivial closed subset $F \subseteq \T$ with nonempty interior, and let $K$ denote the kernel of the corresponding quotient *-homomorphism
        \[ I \rtimes_{\alpha,\lambda} \Gamma \isoto I \tensor \rC(\T) \isoto \rC(\T, I) \surjectsonto \rC(F, I). \]
        We claim that any nonzero element $b \in K$ must have infinitely many nonzero Fourier coefficients. Choose nonzero $b \in K$ and write $b \sim \sum_{k \in \Gamma} b_k \lambda_k$ in $I \rtimes_{\alpha,\lambda} \Gamma$, which corresponds to $b \sim \sum_{k \in \Gamma} b_k v^k e^{ik\theta}$ in $\rC(F,I)$, and suppose for the sake of contradiction that at most finitely many $b_k$'s are nonzero. Since $b \in K$, $\sum_{k \in \Gamma} b_k v^k e^{ik\theta} = 0$ for $\theta \in F$. Since $F$ has non-empty interior, it is easy to check that this implies $b_k v^k = 0$ for all $k \in \Gamma$. Hence $b = 0$, which gives a contradiction.

        In summary, we have shown that any nonzero element in $K$ has infinitely many nonzero Fourier coefficients. An argument similar to the one made above implies that $A + K$ is an intermediate C*-algebra for the inclusion $A \subseteq A \rtimes_{\alpha,\lambda} \Gamma$ that is not the reduced crossed product of a partial subaction of $\alpha$.
    \end{proof}

    We now proceed to the case when every non-trivial element in $\Gamma_K$ is of order $2$. We will require the following preliminary result, which highlights an interesting property of *-automorphisms of order $2$. Namely, if they are quasi-inner, then they restrict to an inner *-automorphism on an essential ideal. 

    \begin{lemma}
    \label{lem:order_two_inner}
        Let $\alpha : A \to A$ be a *-automorphism of a unital C*-algebra $A$ satisfying $\alpha^2 = \id_A$. If $\alpha$ is quasi-inner, then there is an essential $\alpha$-invariant ideal $I \triangleleft A$ with the property that $\alpha$ is inner on $I$, in the sense that $\alpha|_I = \Ad_u$ for some unitary $u$ in the multiplier algebra $M(I)$. Furthermore, we may require that $u^2 = 1$.
    \end{lemma}

    \begin{proof}
        Several steps in the proof will be reminiscent of steps in the proof of Proposition~\ref{prop:counterexample_order_three}, so we will not go into as much detail here. 
        
        First, recall that $\alpha$ being quasi-inner on $A$ means that the unique extension of $\alpha$ to the injective envelope $I(A)$, which we will continue to denote by $\alpha$ for convenience, is inner, i.e. $\alpha = \Ad_u$ for a unitary $u \in I(A)$. Since $\Ad_{u^2} = \alpha^2 = \id$, $u^2 \in Z(I(A))$. Since $Z(I(A))$ is monotone complete and commutative, we can replace $u$ by a unitary in $I(A)$ satisfying $u^2 = 1$.

        We obtain a *-homomorphism $\pi : I(A) \rtimes_{\alpha,\lambda} \Z/2\Z \to I(A)$ satisfying $\pi|_{I(A)} = \id_{I(A)}$ and $\pi(\lambda_1) = u$, where we have written $\Z/2\Z = \{0,1\}$. Let $K = \ker \pi$. Clearly $K \ne 0$, and therefore and so the intersection $J = K \cap (A \rtimes_{\alpha,\lambda} \Z/2\Z)$ must also satisfy $J \ne 0$.

        Note that $J \cap A = 0$. In particular, there is an element in $J$ of the form $a + b \lambda_1$ with $b \ne 0$. Applying $\pi$ yields that $a + bu = 0$, and in particular, that $bu \in A$. Let
        \[ I = \setbuilder{b \in A}{bu \in A}. \]
        Then $I$ is clearly closed under left multiplication by elements in $A$. It is also closed under right multiplication by elements in $A$ since $bau = bu \alpha^{-1}(a)$ for $a \in A$ and $b \in I$. Hence $I$ is an ideal in $A$.

        There are additional properties of $I$ that are important. First, we claim that if $b \in I$, then $ub \in A$. This claim follows from having $b^*u \in A$, and then taking the adjoint, which gives $ub = u^*b \in A$. Second, we claim that $I$ is $\alpha$-invariant. To see this claim, note that for any $b \in I$, we have
        \[ \alpha(b)u = u^2\alpha(b)u = ub \in A. \]
        Hence, $\alpha(I) \subseteq I$. We can deduce that $\alpha(I) = I$ holds from the fact that $I \subseteq \alpha^{-1}(I)$, and the fact that $\alpha$ is of order two, meaning $\alpha^{-1} = \alpha$. Additionally, arguing as in the proof of Proposition~\ref{prop:counterexample_order_three} yields $Iu \subseteq I$ and $uI \subseteq I$.

        Finally, we will show that $I$ is an essential ideal. Supposing otherwise, let $L_0 = I^{\perp} \neq 0$ and note that $L_0$ is also $\alpha$-invariant. Let $p$ denote the central support projection for $L_0$ in $I(A)$ and let $L = \ca(L,p)$ be the minimal unitization of $L$. Then $L$ is a unital C*-algebra and the restriction $\alpha|_L$ is quasi-inner since $I(L) = I(A)p$ and $\alpha|_{I(L)} = \Ad_{up}$. Consequently, there is a nontrivial $\alpha$-invariant ideal $M \triangleleft L$ which is invariant under left/right multiplication by $up$.
        
        By essentiality of $L_0$ in $L$, the ideal $M \cap L_0 = ML_0 = L_0M$ is non-trivial and closed under left and right multiplication by $up$. However, any nonzero element in $M \cap L_0$ does not belong to $I$ since $L_0 = I^\perp$, contradicting the definition of $I$. It follows that $I$ is essential. 

        Applying \cite[Theorem 8.1.29]{SWMonotoneComplete}, we conclude that $u \in M(I)$.
    \end{proof}

    \begin{proposition}
        \label{prop:counterexample_klein_four}
        If every non-trivial element in $\Gamma_K$ is of order $2$ but $\card{\Gamma_K} \geq 3$, then there is an intermediate C*-algebra for the inclusion $A \subseteq A \rtimes_{\alpha,\lambda} \Gamma$ that is not the reduced crossed product of a partial subaction of $\alpha$. 
    \end{proposition}

    \begin{proof}
        It is not hard to prove that a group with the property that every non-trivial element has order $2$ is abelian. Hence $\Gamma_K$ is abelian. Let $s,t \in \Gamma_K$ be distinct nontrivial elements, and observe that they generate a subgroup isomorphic to $\Z/2\Z \times \Z/2\Z$. There, by appealing to Lemma~\ref{lem:intermediate_subalgebra_reductions} again, we may assume without loss of generality that $\Gamma = \Z/2\Z \times \Z/2\Z$, and that the action $\alpha$ is pointwise quasi-inner, i.e.\ $\alpha_g$ is quasi-inner for every $g \in \Gamma$.

        Applying Lemma~\ref{lem:order_two_inner}, we obtain essential ideals $I_s, I_t \triangleleft A$ with the property that for each $g \in \set{s,t}$, $I_g$ is $\alpha_g$-invariant, and $\alpha_g|_{I_g} = \Ad_{u_g}$ for some unitary $u_g \in M(I_g)$ satisfying $u_g^2 = 1$. Consider the intersection
        \[ J = (I_s \cap I_t) \cap \alpha_s(I_s \cap I_t) \cap \alpha_t(I_s \cap I_t) \cap \alpha_{st}(I_s \cap I_t). \]
        This is an intersection of finitely many essential ideals of $A$, and thus is an $\alpha$-invariant essential ideal. Moreover, for $g \in \set{s,t}$, since $J \subseteq I_g$ is an inclusion of essential ideals of $A$, there is an inclusion $M(I_g) \subseteq M(J)$. In particular, for $g \in \set{s,t}$, $u_g \in M(J)$ and $\alpha_g|_J = \Ad_{u_g}$. It follows from this that $\alpha_{st}|_J = \Ad_{u_s u_t}$. Note also that $J$ is $\alpha$-invariant.
        
        Since $\alpha_{(st)^2} = 1$, $(u_s u_t)^2 = z$ for some $z \in Z(M(J))$. We claim that $z^2 = 1$. To see this, first observe that $z = (u_s u_t)^2 = \alpha_s(u_t) u_t$, so we can write $\alpha_s(u_t) = z u_t$. Hence
        \[ u_t = u_s \alpha_s(u_t) u_s = u_s (z u_t) u_s = z \alpha_s(u_t) = z^2 u_t, \]
        giving $z^2 = 1$. Therefore, there is a projection $p \in Z(M(J))$ such that $zp = p$ and $z(1-p) = -(1-p)$. 

        The projections $p$ and $1-p$ are $\alpha$-invariant. Furthermore, $Jp$ is an essential ideal of $M(J)p$. Since $u_s p, u_t p \in M(J)p$, it follows that $u_s p, u_t p \in M(Jp)$. Similarly, $J(1-p)$ is an essential ideal of $M(J)(1-p)$ with $u_s (1-p), u_t (1-p) \in M(J(1-p))$.
        
        Suppose first that $p \ne 0$. Then letting $v_e = p$, $v_s = u_s p$, $v_t = u_t p$ and $v_{st} = u_s u_t p$, we obtain a unitary representation $v : \Z/2\Z \times \Z/2\Z \to U(M(Jp))$. From here, we can argue as in the proof of Proposition \ref{prop:counterexample_order_three} to obtain an intermediate C*-algebra for the inclusion $A \subseteq A \rtimes_{\alpha,\lambda} \Gamma$ that is not the reduced crossed product of a partial subaction of $\alpha$.

        Otherwise, if $p = 0$, then $1-p = 1$, and the unitaries $u_s$ and $u_t$ anticommute, i.e. $u_s u_t = - u_t u_s$. It is straightforward to verify that there is a *-isomorphism $\sigma : M(J) \rtimes_{\alpha|_{M(J)},\lambda} \Z/2\Z \times \Z/2\Z \to M_2(M(J))$ satisfying
        \[ \sigma(a) = \begin{bmatrix} a & 0 \\ 0 & a \end{bmatrix} \quad \text{for} \quad a \in M(J) \]
        and
        \[ \sigma(\lambda_s) = \begin{bmatrix} u_s & 0 \\ 0 & -u_s \end{bmatrix}, \quad \sigma(\lambda_t) = \begin{bmatrix} 0 & u_t \\ u_t & 0 \end{bmatrix}. \]
        Let
        \[ q = \frac{1}{4} \begin{bmatrix} 1 & \sqrt{3} \\ \sqrt{3} & 3 \end{bmatrix} \in M_2(\C) \subseteq M_2(M(J)). \]
        Then $q$ is a projection that commutes with $\sigma(M(J))$. Hence
        \[ r = \sigma^{-1}(q) \in M(J) \rtimes_{\alpha|_{M(J)},\lambda} \Gamma \subseteq A \rtimes_{\alpha,\lambda} \Gamma \]is a projection that commutes with $M(J)$. In fact,
        \[ r = \frac{1}{4} (2 \lambda_e - u_s^* \lambda_s + \sqrt{3} u_t^* \lambda_t). \] 
        Let $B$ denote the norm closure of $A + Jr$. It is easy to verify that $B$ is a C*-algebra, and hence an intermediate C*-algebra for the inclusion $A \rtimes_{\alpha,\lambda} \Gamma$. For $b \in A + Jr$ and $g \in \Z/2\Z \times \Z/2\Z$,
        \[ \sqrt{3} E(b \lambda_s^*)u + E(b \lambda_t^*) v = 0, \]
        and so this is also true for $b \in B$. This implies that the $s$-th Fourier coefficient of $b$ is zero if and only if the $t$-th Fourier coefficient of $b$ is zero. It is clear that there are elements in $B$ with nonzero $s$-th Fourier coefficient. Hence $B$ is not the reduced crossed product of a partial subaction of $\alpha$.
    \end{proof}

    We have now collected the results we need to prove Theorem~\ref{thmintro:converse}, which claims that if there is a proper ideal $K \triangleleft A$ such that the corresponding rigid stabilizer $\Gamma_K$ satisfies $\card{\Gamma_K} \geq 3$, then there is an intermediate C*-algebra for the inclusion $A \rtimes_{\alpha,\lambda} \Gamma$ that is not the reduced crossed product of a partial subaction of $\alpha$. 

    \begin{proof}[Proof of Theorem~\ref{thmintro:converse}]
        If there is at least one element $t \in \Gamma_K$ of order at least $3$, then Proposition~\ref{prop:counterexample_order_three} yields the desired intermediate C*-algebra. Otherwise, if there is no element of order at least $3$, meaning that every nontrivial element of $\Gamma_K$ has order $2$, then the desired intermediate C*-algebra is provided by Proposition~\ref{prop:counterexample_klein_four}.
    \end{proof}

    \subsection{Noncommutative counterexamples}
    \label{sec:noncommutative_counterexamples}
    
    Theorem \ref{thmintro:converse} provides a (nearly) complete converse to Theorem \ref{thmintro:main}. The only case missing is when every non-trivial rigid stabilizer for the action is of order $2$. In this section, we will demonstrate that Theorem \ref{thmintro:converse} fails in this case.

    \begin{proposition}
    \label{prop:Z4_no_counterexample_subalgebras}
        Let $\Gamma = \Z/4\Z = \{0,1,2,3\}$ and let $\alpha : \Gamma \curvearrowright A$ be an action by $\Gamma$ on a simple unital C*-algebra $A$ by *-automorphisms such that $\alpha_1$ is properly outer and $\alpha_2$ is inner, say $\alpha_2 = \Ad_u$ for a unitary $u \in A$ satisfying $\alpha_1(u) = -u$. Then every intermediate C*-algebra for the inclusion $A \subseteq A \rtimes_{\alpha,\lambda} \Gamma$ is of the form $A \rtimes_{\alpha|_\Lambda, \lambda} \Lambda$ for a subgroup $\Lambda \subseteq \Gamma$.
    \end{proposition}

    \begin{proof}
        Let $B$ be an intermediate C*-algebra for the inclusion 
        $A \subseteq A \rtimes_{\alpha,\lambda} \Gamma$, and suppose for the sake of contradiction that $B$ is not of the form $B = A \rtimes_{\alpha|_\Lambda, \lambda} \Lambda$ for a subgroup $\Lambda \subseteq \Gamma$.
        
        First, we claim that there is an element of the form $\sum_{k=0}^3 a_k \lambda_k \in B$ for $a_k \in A$ such that either $a_1 \ne 0$ or $a_3 \ne 0$. Assume otherwise, so that $B \subseteq A + A \lambda_2$. Then because $B \ne A$, there is an element of the form $a_0 + a_2 \lambda_2 \in B$ with $a_2 \ne 0$. Then subtracting off $a_0$ gives $a_2 \lambda_2 \in B$. From here, because $A$ is simple, it is easy to verify that $\lambda_2 \in B$, and it follows that $B = A + A\lambda_2 \isoto A \rtimes_{\alpha|_\Lambda,\lambda} \Lambda$, where $\Lambda = \langle 2 \rangle$, contradicting the assumption on $B$.

        Thus, we can choose $b = \sum_{k=0}^3 b_k \lambda_k \in B$ be an element with either $b_1 \neq 0$ or $b_3 \neq 0$. By subtracting off $b_0$, we can assume that $b_0 = 0$. We claim that $b_2 \lambda_2 \in B$. Indeed, a close examination of the proofs of 
        Proposition~\ref{prop:average_all_to_expectation} and Corollary~\ref{cor:average_all_to_any_coefficient} shows that if $\alpha_1$ (and hence $\alpha_1^{-1}$) are properly outer, then the facts that $\alpha_1^{-1} = \alpha_2 \alpha_3^{-1}$ and $\alpha_1 = \alpha_3 \alpha_2^{-1}$ yields $b_2 \lambda_2 \in B$. Hence we can further assume that $b_2 = 0$.

        If $b_1 = 0$, then $b = b_3 \lambda_3$, and in this case, since $A$ is simple and $3$ generates $\Z/4\Z$, it would follow that $B = A \rtimes_{\alpha,\lambda} \Gamma$, which would be a contradiction. Hence $b_1 \ne 0$. Similarly, $b_3 \ne 0$. Applying the simplicity of $A$ once again, we can further assume that $b_1 = 1$.

        For a unitary $w \in A$, consider the conjugation
        \[ w b \alpha_1^{-1}(w^*) = w \lambda_1 \alpha_1^{-1}(w^*) + w b_3 \lambda_3 \alpha_1^{-1}(w^*) = \lambda_1 + w b_3 \alpha_2(w^*) \lambda_3 \in B. \]
        If $w b_3 \alpha_2(w^*) \ne b_3$, then by subtracting $b$, we would obtain an element of the form $a \lambda_3 \in B$ for $0 \ne a \in A$, and we have already ruled out this possibility above. Therefore, for every unitary $w \in A$, $b_3 \alpha_2(w^*) = w^* b_3$, which is equivalent to $b_3 w = \alpha_2^{-1}(w) b_3$. Since $A$ is spanned by unitaries, this implies that for $a \in A$
        \[ b_3 a = \alpha_2^{-1}(a) b_3,\]
        and similarly that
        \[ b_3^* a = \alpha_2(a) b_3^*. \]
        It follows from this that $b_3^*b_3$ and $b_3 b_3^*$ belong to the center $Z(A)$ of $A$. However, $Z(A) = \C$ since $A$ is simple, so it follows that $b_3$ is a nonzero scalar multiple of a unitary $v \in A$ satisfying $\alpha_2^{-1} = \Ad_v$ and $\alpha_2 = \Ad_{v^*}$.

        By assumption, $\alpha_2 = \Ad_u$. Hence for $a \in A$, $a = \alpha_2^2(a) = v^*uau^*v$, implying $v^*ua = av^*u$. This implies $v^*u \in Z(A) = \C$, so $v$ is actually a nonzero scalar multiple of $u$. Hence $b = \lambda_1 + \beta u \lambda_3$ for $0 \ne \beta \in \C$. Taking the adjoint and using the fact that $\alpha(u) = -u$ gives
        \[ b^* = \lambda_3 + \overline{\beta} \alpha(u)^* \lambda_1 = \lambda_3 - \overline{\beta} u^* \lambda_1.\]
        Thus
        \[ b - \beta u b^* = (1 + |\beta|^2) \lambda_1 \in B, \]
        implying $\lambda_1 \in B$. Subtracting $\lambda_1$ from $b$ yields $b_3 \lambda_3 \in B$. But as we have already argued above, this is not possible, so we obtain the desired contradiction.
    \end{proof}

    \begin{example}
    \label{ex:Z4_no_counterexample_nontrivial}
        The above class of examples in Proposition~\ref{prop:Z4_no_counterexample_subalgebras} is not empty. Indeed, it follows from \cite[Proposition~1.6]{connes_hyperfinite_II1_automorphisms} that the hyperfinite $II_1$ factor $R$ admits an automorphism $\alpha$ with the required properties. Moreover, since $R$ is a $II_1$ factor, it follows from Dixmier's averaging property that $R$ is simple as a C*-algebra. We can even find a separable example by replacing $R$ with a separable C*-subalgebra using Blackadar's notion of \emph{separably inheritable} properties of C*-algebras. Specifically, since simplicity and $\alpha$-invariance are separably inheritable properties, there is a norm-separable C*-subalgebra $A \subseteq R$ that is simple, unital, $\alpha$-invariant, contains $u$ and is weak*-dense in $R$. See \cite[II.8.5.1-II.8.5.6]{blackadar_operator_algebras} for the relevant notions and results. Then note that $\alpha|_A$ is still properly outer.
    \end{example}

    \subsection{The commutative case}
    \label{sec:counterexample_constructions:commutative}

    In this section we will prove Theorem~\ref{thmintro:commutative}, which asserts that the complete converse of Theorem \ref{thmintro:main} holds for commutative C*-algebras. 

    Fix an action $\theta : \Gamma \curvearrowright X$ of a discrete group $\Gamma$ on a compact Hausdorff space $X$ and let $\alpha : \Gamma \curvearrowright \rC(X)$ denote the corresponding action of $\Gamma$ on the unital commutative C*-algebra $\rC(X)$ by *-automorphisms. Recall that the (pointwise) freeness of $\theta$ is equivalent to the pointwise residual proper outerness of $\alpha$.
    
    Suppose that $|\Gamma| \geq 3$ and that $\theta$ is not free, so that $\alpha$ is not pointwise residually properly outer. If there is a proper ideal $K \triangleleft \rC(X)$ such that the corresponding rigid stabilizer $\Gamma_K$ satisfies $\card{\Gamma_K} \geq 3$, then the results in Section \ref{sec:counterexample_constructions:noncommutative} yields an intermediate C*-algebra for the inclusion $\rC(X) \subseteq \rC(X) \rtimes_{\alpha,\lambda} \Gamma$ that is not the reduced crossed product of a partial subaction of $\alpha$. 

    Otherwise, it must be the case that for every proper ideal $K \triangleleft \rC(X)$ such that the corresponding rigid stabilizer $\Gamma_K$ is non-trivial, $\card{\Gamma_K} = 2$. Then in particular, for every point $x \in X$ such that the stabilizer $\Gamma_x = \{g \in \Gamma : \theta_g(x) = x\}$ is non-trivial, $\card{G_x} = 2$. The following result asserts that, in this situation, there is an intermediate C*-algebra for the inclusion $\rC(X) \subseteq \rC(X) \rtimes_{\alpha,\lambda} \Gamma$ that is of the desired form.
    
    \begin{proposition}
    \label{prop:converse_two_stabilizers}
        If there is $x \in X$ such that the corresponding stabilizer $\Gamma_x$ satisfies $\card{\Gamma_x} = 2$, then there is an intermediate C*-algebra for the inclusion $\rC(X) \subseteq \rC(X) \rtimes_{\alpha,\lambda} \Gamma$ that is not the reduced crossed product of a partial subaction of $\alpha$.
    \end{proposition}

    \begin{proof}
        For the proof, it will be convenient to omit $\theta$ and write $\theta_g(x)$ as $gx$ for $g \in \Gamma$ and $x \in X$.
    
        Choose $r,s \in \Gamma \setminus \set{e}$ with the property that $r \neq s$, and $sr^{-1}$ is the nontrivial element of $\Gamma_x$. Equivalently, such that $r$ and $s$ are distinct elements of the same non-trivial left coset of $\Gamma_x \backslash \Gamma$. Note that $r^{-1}x = s^{-1}x$, $sr^{-1} = rs^{-1}$ and $r^{-1} s = s^{-1} r$. Also, $sx \ne x$ and $rx \ne x$.

        The key observation underlying the proof is that if $b = f_r \lambda_r + f_s \lambda_s \in \rC(X) \rtimes_{\alpha,\lambda} \Gamma$ for $f_r,f_s \in \rC(X)$ with the property that $f_r(x)$ and $f_s(x)$ are nonzero and equal, then it is impossible to obtain an element with nonzero $r$-th Fourier coefficient but zero $s$-th Fourier coefficient (and vice-versa) from $b$ using $\rC(X)$-bimodule operations. The difficult part is constructing a C*-algebra based on this observation.
        
        Let $U$ be an open neighbourhood of $x$ such that $rU \cap U = \emptyset$ and $sU \cap U = \emptyset$. By replacing $U$ with $U \cap rs^{-1} U$, we can further assume that $r^{-1} U = s^{-1} U$.

        Define $B$ by
        \[ B = \{f_1 \lambda_{r^{-1}s} + (f_2 \lambda_{r^{-1}} + f_3 \lambda_{s^{-1}}) + f_4 + (f_5 \lambda_r + f_6 \lambda_s) + f_7 \lambda_{rs^{-1}}\}, \]
        for $f_1,f_2,f_3,f_4,f_6,f_6,f_7 \in \rC(X)$ such that
        \begin{enumerate}
            \item $f_1$ is supported on $r^{-1}U = s^{-1}U$,
            \item $f_2$ and $f_3$ are supported on $r^{-1}U = s^{-1}U$ and satisfy $f_2(r^{-1}x) = f_3(s^{-1}x)$,
            \item $f_4$ has no restrictions
            \item $f_5$ and $f_6$ are supported on $U$ and satisfy $f_5(x) = f_6(x)$, and
            \item $f_7$ is supported on $U$.
        \end{enumerate}
        It is clear that $B$ is a subspace of $\rC(X) \rtimes_{\alpha,\lambda} \Gamma$. We will show that $B$ is an intermediate C*-algebra for the inclusion $\rC(X) \subseteq \rC(X) \rtimes_{\alpha,\lambda} \Gamma$ that is not the reduced crossed product of a partial subaction of $\alpha$. 

        It is not difficult to verify that $B$ is an operator system, i.e. that it is closed under the adjoint. To see that it is a *-algebra, first observe that above conditions on the $f_i$'s are closed under left multiplication by elements in $\rC(X)$. Hence $B$ is closed under left multiplication by elements in $\rC(X)$, and by the self-adjointness of $B$, under right multiplication by elements in $\rC(X)$.
        
        Next, suppose that $h_1,h_2,h_3,h_4,h_5,h_6,h_7 \in \rC(X)$ also satisfy the above conditions. Then
        \begin{align*}
            (f_5 \lambda_r + f_6 \lambda_s)(h_5 \lambda_r + h_6 \lambda_s) &= 0, \\
            (f_5 \lambda_r + f_6 \lambda_s) (h_7 \lambda_{rs^{-1}}) &= 0, \\
            (f_1 \lambda_{r^{-1}s}) (h_5 \lambda_r + h_6 \lambda_s) &= 0, \\
            (f_1 \lambda_{r^{-1}s}) (h_7 \lambda_{rs^{-1}}) &= 0,
        \end{align*}
        while on the other hand,
        \begin{align*}
            (f_2 \lambda_{r^{-1}} + f_3 \lambda_{s^{-1}}) (h_5 \lambda_r +      h_6 \lambda_s) &= k_4 + k_1 \lambda_{r^{-1}s}, \\
            (f_5 \lambda_r + f_6 \lambda_s) (h_2 \lambda_{r^{-1}} + h_3         \lambda_{s^{-1}}) &= k_4 + k_7 \lambda_{rs^{-1}}, \\
            (f_7 \lambda_{rs^{-1}}) (h_5 \lambda_r + h_6 \lambda_s) &= k_5      \lambda_r + k_6 \lambda_s, \\
            (f_5 \lambda_r + f_6 \lambda_s) (h_1 \lambda_{r^{-1}s}) &= k_5      \lambda_r + k_6 \lambda_s, \\
            (f_1 \lambda_{r^{-1}s})(h_1 \lambda_{r^{-1}s}) &= k_4,\\
            (f_7 \lambda_{rs^{-1}})(h_7 \lambda_{rs^{-1}}) &= k_4,
        \end{align*}
        for $k_1,k_2,k_3,k_4,k_5,k_6,k_7 \in \rC(X)$ satisfying the above conditions. Note that we are not claiming that the same values of $k_i$ are obtained for distinct lines above. Applying the self-adjointness of $B$ again, we conclude that $B$ is a *-algebra.

        It is clear that $B$ is an intermediate *-algebra for the inclusion $\rC(X) \subseteq \rC(X) \rtimes_{\alpha,\lambda} \Gamma$. It remains to show that $B$ is closed, but not the reduced crossed product of a partial subaction of $\alpha$.
        
        Note that in the definition of $B$, we have not ruled out the possibility that distinct formal words in $\{r,r^{-1},s,s^{-1}, r^{-1}s, rs^{-1}\}$ evaluate to the same element in $\Gamma$. However, it is easy to check that, writing $b \in B$ as $b = \sum_{i=1}^7 f_i \lambda_{g_i}$, and canonically viewing $f_i \lambda_{g_i} \in C_c(X \times G)$, then the supports of each $f_i \lambda_{g_i}$ on $X \times G$ are disjoint.
        
        It follows from the previous paragraph that the convergence of a net $(b^{(k)})$ of elements in $B$ is equivalent to the convergence of the corresponding net of functions $(f_i^{(k)})$ from the definition of $B$ for $1 \leq i \leq 7$, and it follows easily from this that $B$ is closed. Hence $B$ is an intermediate C*-algebra. Further, it follows from the previous paragraph and the definition of $B$ that if $b \in B$ has a nonzero Fourier coefficient at $r$, then it also has a nonzero Fourier coefficient at $s$. Therefore, $B$ is not the reduced crossed product of a partial subaction of $\alpha$.
    \end{proof}

   We are now ready to prove Theorem~\ref{thmintro:commutative}, which establishes a necessary and sufficient condition for every intermediate C*-algebra for the inclusion $\rC(X) \subseteq \rC(X) \rtimes_{\alpha,\lambda} \Gamma$ to be the reduced crossed product of a partial subaction of $\alpha$.

    \begin{proof}[Proof of Theorem~\ref{thmintro:commutative}]
        If $\Gamma$ has the approximation property and either $\card{\Gamma} \geq 2$ or $\theta$ is free, then every intermediate C*-algebra for the inclusion $\rC(X) \subseteq \rC(X) \rtimes_{\alpha,\lambda} \Gamma$ is the reduced crossed product of a partial subaction of $\alpha$ by Theorem \ref{thmintro:main}.

        Conversely, suppose $\card{\Gamma} \geq 3$ and $\theta$ is not free. Then $\alpha$ is not pointwise residually properly outer. If there is $x \in X$ such that the stabilizer $\Gamma_x$ satisfies $\card{\Gamma_x} \geq 3$, then Theorem \ref{thmintro:converse} implies the existence of an intermediate C*-algebra for the inclusion $\rC(X) \subseteq \rC(X) \rtimes_{\alpha,\lambda} \Gamma$ that is not the reduced crossed product of a partial action of $\alpha$. Otherwise, since $\theta$ is not free, $\card{\Gamma_x} = 2$ for some $x \in X$, and an intermediate C*-algebra that is not the reduced crossed product of a partial subaction of $\alpha$ is provided by Proposition~\ref{prop:converse_two_stabilizers}.
    \end{proof}

	\bibliographystyle{amsalpha}
	\bibliography{intermediate_subalgebras}
\end{document}